\documentclass[11pt]{amsart}

\usepackage{mathrsfs}
\usepackage{hyperref}
\usepackage[svgnames]{xcolor}
\hypersetup{colorlinks,breaklinks,
            linkcolor=DarkBlue,urlcolor=DarkBlue,
            anchorcolor=DarkBlue,citecolor=DarkBlue}
\usepackage{url}
\usepackage{tikz}
\usetikzlibrary{decorations.pathreplacing,shapes.geometric}
\usepackage{graphicx}
\usepackage{euscript}
\usepackage[verbose,letterpaper,tmargin=1in,bmargin=1in,lmargin=1in,rmargin=1in]
{geometry}
\usepackage{amssymb}
\usepackage{pst-plot}

\newtheorem{theorem}{Theorem}[section]

\newtheorem{proposition}[theorem]{Proposition}
\newtheorem{corollary}[theorem]{Corollary}
\newtheorem{lemma}[theorem]{Lemma}

\theoremstyle{definition}

\newtheorem{remark*}[theorem]{}

\theoremstyle{remark}
\newtheorem{remark}[theorem]{Remark}

\newcommand{\C}{\mathbb C}
\newcommand{\Z}{\mathbb Z}

\DeclareMathOperator{\Hom}{Hom}
\DeclareMathOperator{\Der}{Der}

\newcommand{\mc}[1]{\mathcal{#1}}
\newcommand{\mb}[1]{\mathbb{#1}}
\newcommand{\frk}[1]{\mathfrak{#1}}

\tikzset{Box/.style={very thick, rounded corners}}
\tikzset{marked/.style={star, star point height = .75mm, star points =5, fill=black,minimum size=2mm, inner sep=0mm} }
\tikzset{verythickline/.style = {line width=5pt}}
\tikzset{thickline/.style = {line width=4pt}}
\tikzset{medthick/.style = {line width=3pt}}
\tikzset{count/.style = {fill=white,circle,draw,thin, inner sep=2pt}}
\tikzset{rcount/.style = {fill=white,rectangle,draw,thin,inner sep=2pt, rounded corners}}
\tikzset{flip/.style = {rotate=0}}
\begin{document}

\title{On the symmetric enveloping algebra of planar algebra subfactors}
\author{Stephen Curran}\thanks{S.C.: Research supported by an NSF postdoctoral fellowship and NSF grant DMS-0900776.}
\address{S.C.: Department of Mathematics, UCLA, Los Angeles, CA 90095.}
\email{\href{mailto:curransr@math.ucla.edu}{curransr@math.ucla.edu}}
\author{Vaughan F. R. Jones}\thanks{V.J.: Research supported by NSF grant DMS-0856316.}
\address{V.J.: Department of Mathematics, University of California, Berkeley, CA 94720.}
\email{\href{mailto:vfr@math.berkeley.edu}{vfr@math.berkeley.edu}}
\author{Dimitri Shlyakhtenko}\thanks{D.S.: Research supported by NSF grant DMS-0900776.\\
\indent S.C., V.J. and D.S. are also supported by DARPA Award 0011-11-0001.}
\address{D.S.: Department of Mathematics, UCLA, Los Angeles, CA 90095.}
\email{\href{mailto:shlyakht@math.ucla.edu}{shlyakht@math.ucla.edu}}

\begin{abstract}
We give a diagrammatic description of Popa's symmetric enveloping algebras associated to planar algebra subfactors. As an application we construct a natural family of derivations on these factors, and compute a certain free entropy dimension type quantity.
\end{abstract}

\maketitle

\section*{Introduction}

Inspired by Voiculescu's \cite{voi1} description of the large $N$ limit of the distribution of families of independent random matrices, Guionnet, Jones and Shlyakhtenko \cite{gjs1} introduced diagrammatic traces on a tower of graded algebras $Gr_0(\mc P) \subset Gr_1(\mc P) \subset \dotsb$ which is naturally associated to a subfactor planar algebra $\mc P$.  The main result of \cite{gjs1} is that the completions form a tower of II$_1$ factors $M_0 \subset M_1 \subset \dotsb$ whose standard invariant is $\mc P$, thus giving a diagrammatic proof of a fundamental result of Popa \cite{pop2}.  In subsequent work \cite{gjs2} they were able to identify the isomorphism class of these factors in the case that $\mc P$ is finite-depth.  Explicitly, $M_k$ is isomorphic to an interpolated free group factor $L(\mb F_{r_k})$, where
\begin{equation*}
 r_k = 1 + 2\delta^{-2k}I(\delta-1). \label{param} \tag{$*$}
\end{equation*}
Here $\delta^2$ is the index $[M_1:M_0]$, see Section \ref{sec:enveloping} for the definition of the global index $I$.  The fact that $M_k$ is an interpolated free group factor was also obtained independently by Kodiyalam and Sunder \cite{ks2}.  See also \cite{gjsz},\cite{jsw},\cite{ks1}.

Given an inclusion of II$_1$ factors $N \subset M$, Popa \cite{pop1} has shown that there is a unique II$_1$ factor $M \boxtimes_{e_N} M^{op}$, called the \textit{symmetric enveloping algebra} of the inclusion, which is generated by a copy of $M \otimes M^{op}$ and a projection $e_N$ which is simultaneously the Jones projection for the inclusions $N \subset M$ and $N^{op} \subset M^{op}$.  This algebra, together with the inclusion $M \otimes M^{op} \subset M \boxtimes_{e_N} M^{op}$, encodes a number of important analytic properties of the subfactor $N \subset M$, such as amenability and property T \cite{pop3}.  When the inclusion $N \subset M$ is \textit{strongly amenable}, for example if $M$ is hyperfinite and $N \subset M$ is finite-depth, then $M \otimes M^{op} \subset M \boxtimes_{e_N} M^{op}$ is isomorphic to Ocneanu's \textit{asymptotic inclusion} \cite{ocn}.  

In this paper we give a diagrammatic construction of the symmetric enveloping algebras of the inclusions $M_k \subset M_{k+1}$.  (In the case that $\mc P$ is the Temperley-Lieb planar algebra, this was announced in \cite{shl1}).  Using this description, we give a simple diagrammatic proof that the index $[M_k \boxtimes_{e_{k-1}} M_k^{op}:M_k \otimes M_k^{op}]$ is the global index $I$, first computed by Ocneanu \cite{ocn}.  As an application we construct a module of closable derivations $\delta:Gr_0 \to Gr_0 \otimes Gr_0^{op}$, related to the Schwinger-Dyson equation from \cite{gjsz}, whose Murray von-Neumann dimension (after completion) is precisely the parameter $r_0$ from $(*)$.  Since the dimensions of such spaces of derivations are well-known to be related to Voiculescu's free entropy dimension (see e.g. \cite{shl2}), this provides a more conceptual understanding of the rather mysterious parameters $r_k$ computed in \cite{gjs2}.

Our paper is organized as follows.  We begin by briefly recalling the construction from \cite{gjs1} in Section \ref{sec:background}.  In Section \ref{sec:semi-finite} we construct an auxiliary semi-finite algebra $\mathfrak M$, starting from a subfactor planar algebra $\mc P$.  Section \ref{sec:enveloping} contains our main result: the symmetric enveloping algebras of the inclusions $M_{k} \subset M_{k+1}$ are naturally isomorphic to certain compressions of $\mathfrak M$.  We then compute the index of $M_{k} \otimes M_{k}^{op} \subset M_{k} \boxtimes_{e_{k-1}} M_{k}^{op}$.  In Section \ref{sec:derivations} we construct derivations on $Gr_0$, and compute a free entropy dimension type quantity for $M_0$.

\section{Planar algebra subfactors}\label{sec:background}
In this section we briefly recall the construction from \cite{gjs1}.  The reader is referred to this paper and to \cite{jsw}, \cite{gjs2} for further details.

Let $\mc P$ be a subfactor planar algebra with graded components $P_k^\epsilon$, $k = 0,1,2,\dotsc$, $\epsilon = \pm$.  For $n,k \geq 0$ and $\epsilon = \pm$ let $P^\epsilon_{n,k}$ be a copy of $P^\epsilon_{n+k}$.  Elements of $P^\epsilon_{n,k}$ will be represented by diagrams
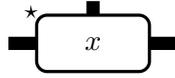
\begin{equation*}
 \begin{tikzpicture}[scale=.75]
 \draw[Box](0,0) rectangle (2,1);
  \draw[verythickline](0,.5) -- (-.5,.5); \draw[verythickline](2,.5) -- (2.5,.5);
\draw[verythickline](1,1) -- (1,1.25);
 \node at (1,.5) {$x$}; \node[marked] at (-.1,1.05) {};
 \end{tikzpicture}\end{equation*}
where the thick lines to the left and right represent $k$ strings, and the thick line at top represents $2n$ strings.  We will typically suppress the marked point $\star$, and take the convention that it occurs at the top-left corner and is shaded according to $\epsilon$.

Define a product $\wedge_k:P^\epsilon_{n,k} \times P^\epsilon_{m,k} \to P^\epsilon_{n+m,k}$ by
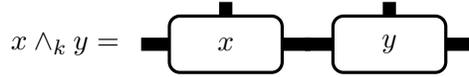
\begin{equation*}
\begin{tikzpicture}[thick,scale=.75]
\draw[Box](0,0) rectangle (2,1);
  \draw[verythickline](0,.5) -- (-.5,.5); \draw[verythickline](2,.5) -- (2.5,.5);
 \draw[verythickline](1,1) -- (1,1.25);
 \node at (1,.5) {$x$}; \node[left] at (-.7,.5) {$x \wedge_k y =$};
\begin{scope}[xshift=2.9cm]
\draw[Box](0,0) rectangle (2,1);
  \draw[verythickline](0,.5) -- (-.5,.5); \draw[verythickline](2,.5) -- (2.5,.5);
 \draw[verythickline](1,1) -- (1,1.25);
 \node at (1,.5) {$y$};
\end{scope}
\end{tikzpicture}
\end{equation*}
The involution $\dagger:P^{\epsilon}_{n,k} \to P^{\epsilon}_{n,k}$ is given by
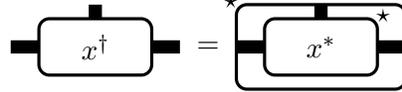
\begin{equation*}
\begin{tikzpicture}[thick,scale=.75]
\draw[Box](0,0) rectangle (2,1);
  \draw[verythickline](0,.5) -- (-.5,.5); \draw[verythickline](2,.5) -- (2.5,.5);
 \draw[verythickline](1,1) -- (1,1.25);
 \node at (1,.5) {$x^\dagger$}; \node at (3,.5) {$=$};
\begin{scope}[xshift=4cm]
\draw[Box](0,0) rectangle (2,1);
  \draw[verythickline](0,.5) -- (-.5,.5); \draw[verythickline](2,.5) -- (2.5,.5);
 \draw[verythickline](1,1) -- (1,1.25);
 \node at (1,.5) {$x^*$};
 \draw[Box] (-.5,-.25) rectangle (2.5,1.25); \node[marked] at (2.1,1.05) {}; \node[marked] at (-.6,1.3) {};
\end{scope}
\end{tikzpicture}
\end{equation*}

The \textit{Voiculescu trace} $\tau_k:P^\epsilon_{n,k} \to \C$ is defined by
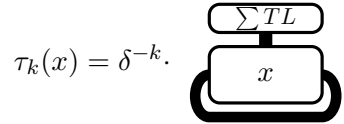
\begin{equation*}
\begin{tikzpicture}[scale=.75]
  \draw[Box] (0,0) rectangle (2,1); \node at (1,.5) {$x$};
  \draw[thickline] (0,.5) arc(90:270:.25cm and .375cm) -- (2,-.25) arc(-90:90:.25cm and .375cm);
  \draw[Box] (0,1.25) rectangle (2,1.75); \node[scale=.8] at (1,1.5) {$\sum TL$};
\draw[verythickline] (1,1) -- (1,1.25);
  \node[left] at (-.5,.75) {$\tau_k(x) = \delta^{-k} \cdot$};
 \end{tikzpicture}
\end{equation*}
where $\sum TL$ denotes the sum over all loopless Temperley-Lieb diagram with $2n$ boundary points.  Let $Gr_k^\epsilon(\mc P) = \bigoplus_{n \geq 0} P^\epsilon_{n,k}$, and observe that the formulas above give $Gr^\epsilon_k(\mc P)$ the structure of a graded $*$-algebra with trace $\tau_k$.  The unit of $Gr^\epsilon_k(\mc P)$ is the element of $P^\epsilon_{0,k}$ consisting of $k$ parallel lines. 

Let $\mathbf{e}_{k,\epsilon}$ denote the following element of $P^\epsilon_{0,k+2}$, $k \geq 0$:
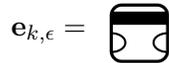
\begin{equation*}
\begin{tikzpicture}
\begin{scope}[xscale=.75,yscale=.5]
 \draw[Box] (0,0) rectangle (1,1.5); \draw[verythickline] (0,1.15) -- (1,1.15);
 \draw[thick] (0,.75) arc(90:-90:.25cm); \draw[thick] (1,.75) arc(90:270:.25cm);
\node[left=2mm,scale=1] at (0,.75) {$\mathbf{e}_{k,\epsilon} = $};
\end{scope}
\end{tikzpicture}
\end{equation*}
Note that there are natural inclusions of $Gr_k^\epsilon(\mc P)$ into $Gr_{k+1}^{\epsilon}(\mc P)$ defined by
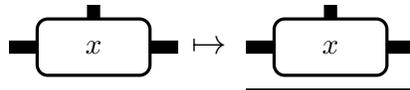
\begin{equation*}
 \begin{tikzpicture}[scale=.75]
 \draw[Box](0,0) rectangle (2,1);
  \draw[verythickline](0,.5) -- (-.5,.5); \draw[verythickline](2,.5) -- (2.5,.5);
\draw[verythickline](1,1) -- (1,1.25);
 \node at (1,.5) {$x$};
\begin{scope}[xshift=4.2cm]
 \draw[Box](0,0) rectangle (2,1);
  \draw[verythickline](0,.5) -- (-.5,.5); \draw[verythickline](2,.5) -- (2.5,.5);
\draw[verythickline](1,1) -- (1,1.25);
\draw[thick] (-.5,-.25) -- (2.5,-.25);
 \node at (1,.5) {$x$};
\node[left=1mm,scale=1.2] at (-.5,.5) {$\mapsto$};
\end{scope}
 \end{tikzpicture}
\end{equation*}
The main result of \cite{gjs1} is the following:

\begin{theorem}
For $k \geq 0$ and $\epsilon = \pm$, the Voiculescu trace $\tau_k$ is a faithful tracial state on $Gr_k^\epsilon(\mc P)$, and its GNS completion is a $II_1$ factor $M_k^\epsilon$ as long as $\delta > 1$.  The inclusions $Gr_k^\epsilon(\mc P) \subset Gr_{k+1}^\epsilon(\mc P)$ extend to $M_k^\epsilon \subset M_{k+1}^\epsilon$, and $(M_{k+1}^\epsilon,\mathbf e_{k,\epsilon})$ is the tower of the basic construction for the subfactor $M_0^\epsilon \subset M_1^\epsilon$.  Moreover, the planar algebra of $M_0 \subset M_1$ is isomorphic to $\mc P$.
\end{theorem} \qed

\section{A semi-finite enveloping algebra associated to a planar algebra}\label{sec:semi-finite}

Let $\mc P$ be a subfactor planar algebra.  For integers $k_1,k_2,s,t$ which add to $2n$, and $\epsilon = \pm$, let $V_{k_1,k_2}^\epsilon(s,t)$ be a copy of $P^\epsilon_{n}$.  Elements of $V^\epsilon_{k_1,k_2}(s,t)$ will be represented by diagrams of the form
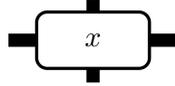
\begin{equation*}
 \begin{tikzpicture}[scale=.75]
 \draw[Box](0,0) rectangle (2,1);
  \draw[verythickline](0,.5) -- (-.5,.5); \draw[verythickline](2,.5) -- (2.5,.5);
 \draw[verythickline](1,0) -- (1,-.25); \draw[verythickline](1,1) -- (1,1.25);
 \node at (1,.5) {$x$};
 \end{tikzpicture}
\end{equation*}
where the thick lines at left, right, top and bottom represent $k_1,k_2,s$ and $t$ strings, respectively.  We take the convention that the marked point occurs at the upper left corner, which is shaded $\epsilon$.  If $\epsilon$ is not specified we assume that $\epsilon = +$.

Define a product $\wedge:V^\epsilon_{k_1,k_2}(s,t) \times V^{\epsilon'}_{k_1',k_2'}(s',t') \to V^\epsilon_{k_1,k_2'}(s+s',t+t')$ to be zero unless $k_2 = k_1'$ and $\epsilon' = (-1)^s \cdot \epsilon$, and otherwise by
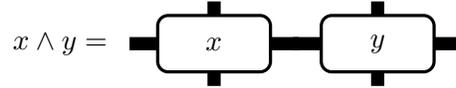
\begin{equation*}
\begin{tikzpicture}[thick,scale=.75]
\draw[Box](0,0) rectangle (2,1);
  \draw[verythickline](0,.5) -- (-.5,.5); \draw[verythickline](2,.5) -- (2.5,.5);
 \draw[verythickline](1,0) -- (1,-.25); \draw[verythickline](1,1) -- (1,1.25);
 \node at (1,.5) {$x$}; \node[left] at (-.7,.5) {$x \wedge y =$};
\begin{scope}[xshift=2.9cm]
\draw[Box](0,0) rectangle (2,1);
  \draw[verythickline](0,.5) -- (-.5,.5); \draw[verythickline](2,.5) -- (2.5,.5);
 \draw[verythickline](1,0) -- (1,-.25); \draw[verythickline](1,1) -- (1,1.25);
 \node at (1,.5) {$y$};
\end{scope}
\end{tikzpicture}
\end{equation*}
The adjoint $\dagger:V_{k_1,k_2}^{\epsilon}(s,t) \to V_{k_2,k_1}^{(-1)^s\epsilon}(s,t)$ is defined by
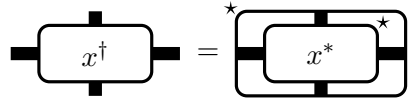
\begin{equation*}
\begin{tikzpicture}[thick,scale=.75]
\draw[Box](0,0) rectangle (2,1);
  \draw[verythickline](0,.5) -- (-.5,.5); \draw[verythickline](2,.5) -- (2.5,.5);
 \draw[verythickline](1,0) -- (1,-.25); \draw[verythickline](1,1) -- (1,1.25);
 \node at (1,.5) {$x^\dagger$};
\node at (3,.5) {$=$};
\begin{scope}[xshift=4cm]
\draw[Box](0,0) rectangle (2,1);
  \draw[verythickline](0,.5) -- (-.5,.5); \draw[verythickline](2,.5) -- (2.5,.5);
 \draw[verythickline](1,0) -- (1,-.25); \draw[verythickline](1,1) -- (1,1.25);
 \node at (1,.5) {$x^*$}; \node[marked] at (2.1,1.05) {};
\draw[Box](-.5,-.25) rectangle (2.5,1.25); \node[marked] at (-.6,1.3) {};
\end{scope}
\end{tikzpicture} 
\end{equation*}
Define $Tr:V^{\epsilon}_{k_1,k_2}(s,t) \to \C$ to be zero unless $k_1 = k_2$ and both $s$ and $t$ are even, and otherwise by
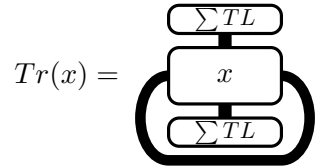
\begin{equation*}
\begin{tikzpicture}[scale=.75]
  \draw[Box] (0,0) rectangle (2,1); \node at (1,.5) {$x$};
  \draw[thickline] (0,.5) arc(90:270:.5cm and .75cm) -- (2,-1) arc(-90:90:.5cm and .75cm);
  \draw[Box] (0,-.25) rectangle (2,-.75); \node[scale=.8] at (1,-.5) {$\sum TL$};
  \draw[Box] (0,1.25) rectangle (2,1.75); \node[scale=.8] at (1,1.5) {$\sum TL$};
  \draw[verythickline](1,0) -- (1,-.25); \draw[verythickline] (1,1) -- (1,1.25);
  \node[left] at (-.7,.5) {$Tr(x) = $};
 \end{tikzpicture}
\end{equation*}
Let
\begin{equation*}
 V = \bigoplus_{\epsilon = \pm} \bigoplus_{\substack{k_1,k_2,s,t \geq 0\\ k_1 + k_2 + s + t\text{ even}}}  V^\epsilon_{k_1,k_2}(s,t).
\end{equation*}

It is easy to see that the formulas above give $V$ the structure of a (non-unital) graded $*$-algebra with trace $Tr$.  
\begin{theorem}\label{bounded}
$Tr$ is a faithful, positive linear functional on $V$.  Moreover, $V$ acts by bounded operators on the GNS Hilbert space for $Tr$.
\end{theorem}

The (semi-finite) von Neumann algebra obtained by completing $V$ under the GNS representation for $Tr$ will be denoted $\mathfrak M$.  To prove Theorem \ref{bounded} we use the orthogonalization method of \cite{jsw}.

\subsection{An orthogonal basis}\label{sec:orthogonal}

We will now define a new multiplication and trace on $V$.  To distinguish between these two structures, we define $W^\epsilon_{k_1,k_2}(s,t)$ to be a copy of $V_{k_1,k_2}^{\epsilon}(s,t)$, and let $W$ be the same direct sum defining $V$.  

Define a product $\star$ on $W^\epsilon_{k_1,k_2}(s,t) \times W^{\epsilon'}_{k_1',k_2'}(s',t')$ to be zero unless $k_2 = k_1'$ and $\epsilon' = (-1)^s\cdot \epsilon$, and otherwise by
\begin{equation*}
 \begin{tikzpicture}[thick,scale=.6]
\draw[Box](0,0) rectangle (2,1);
  \draw[verythickline](0,.5) -- (-.5,.5); \draw[verythickline](2,.5) -- (2.5,.5);
 \draw[verythickline](.5,0) -- (.5,-1); \draw[verythickline](.5,1) -- (.5,2);
 \node at (1,.5) {$x$};
\begin{scope}[xshift=2.9cm]
\draw[Box](0,0) rectangle (2,1);
  \draw[verythickline](0,.5) -- (-.5,.5); \draw[verythickline](2,.5) -- (2.5,.5);
 \draw[verythickline](1.5,0) -- (1.5,-1); \draw[verythickline](1.5,1) -- (1.5,2);
 \node at (1,.5) {$y$};
\end{scope}
\draw[verythickline] (1.5,1) .. controls ++(0,.3) and ++(-.5,0) .. (2.5,1.6) .. controls ++(.5,0) and ++(0,.3) .. (3.5,1);
\node[circle,draw,thin,fill=white,inner sep=2pt,scale=.8] at (2.5,1.6) {$i$};
\draw[verythickline] (1.5,0) .. controls ++(0,-.3) and ++(-.5,0) .. (2.5,-.6) .. controls ++(.5,0) and ++(0,-.3) .. (3.5,0);
\node[circle,draw,thin,inner sep = 2pt, fill=white,scale=.8] at (2.5,-.6) {$j$};

 \node at (-.5,.2) [left=2mm] {$x \star y = \displaystyle\sum_{\substack{1 \leq i \leq \min(s,s')\\ 1 \leq j \leq \min(t,t')}}$};
\end{tikzpicture}
\end{equation*}
where $i$ and $j$ mean that there are $i$ parallel strings at top and $j$ parallel strings at bottom.  The numbers of the other parallel strings are then determined by our conventions.

Define a trace $Tr':W^\epsilon_{k_1,k_2}(s,t)$ to be zero unless $k_1 = k_2$ and $s=t=0$, and otherwise by
\begin{equation*}
\begin{tikzpicture}[scale=.75]
  \draw[Box] (0,0) rectangle (2,1); \node at (1,.5) {$x$};
  \draw[thickline] (0,.5) arc (90:270:.25cm and .375cm) -- (2,-.25) arc(-90:90:.25cm and .375cm);
  \node[left] at (-.45,.3) {$Tr'(x) = $};
 \end{tikzpicture}
\end{equation*}
  
Observe that there is a natural action of $TL \otimes TL^{op}$ on $V$ as follows: if $a, b$ are TL diagrams with $s'$ (resp. $t$) points on top and $s$ (resp. $t'$) points on bottom, then $a \otimes b$ acts as the sum over $k_1,k_2$ and $\epsilon$ of the linear maps $V_{k_1,k_2}^\epsilon(s,t) \to V_{k_1,k_2}^\epsilon(s',t')$ defined by 
\begin{equation*}
\begin{tikzpicture}[scale=.75]
  \draw[Box] (0,0) rectangle (2,.5); \node at (1,.25) {$x$};
  \draw[Box] (0,-.25) rectangle (2,-.75); \node[scale=.9] at (1,-.5) {$b$};
  \draw[Box] (0,.75) rectangle (2,1.25); \node[scale=.9] at (1,1) {$a$};
  \draw[verythickline](1,0) -- (1,-.25); \draw[verythickline](1,1.25) -- (1,1.5); \draw[verythickline] (1,.5) -- (1,.75);
  \draw[verythickline] (1,-.75) -- (1,-1); \node[left=2mm] at (0,.25) {$x \mapsto$};
 \end{tikzpicture}
\end{equation*}

Recall that a TL diagram is called \textit{epi} if each point on the top is connected to a point on the bottom, and \textit{monic} if each point on the bottom is connected to a point on top.  Furthermore, an epi (resp. monic) TL diagram is said to be \textit{nonnested} if each ``turn-back'' on the bottom (resp. top) of the diagram encloses no other turn-backs.

Define $X:V \to W$ to be the sum of all $a \otimes b$ where $a$ (resp. $b$) are epi (resp. monic) TL diagrams.  Likewise, define $Y:W \to V$ to be the sum of all $a \otimes b$ where $a$ (resp. $b$) are nonnested epi (resp. monic) TL diagrams, and the coefficient on $a \otimes b$ is $(-1)^{s+t-s'-t'}$ when $a,b$ are as above.

\begin{proposition}\label{cob}\hfill
\begin{enumerate}
 \item $XY = 1 = YX$.
 \item $X(a \wedge b) = X(a) \star X(b)$.
 \item $X(a^*) = X(a)^*$.
\item $Tr'(X(a)) = Tr(a)$.
\end{enumerate}
\end{proposition}
\begin{proof}
The proof is identical to the argument in \cite[Section 5]{jsw}, and will therefore be omitted.  
\end{proof}

\begin{corollary}
$(W,\star,*)$ is an associative $*$-algebra.
\end{corollary}
\begin{proof}
Since $(V,\wedge,*)$ is clearly associative, this follows from the isomorphism in Proposition \ref{cob}.
\end{proof}

We are now prepared to prove Theorem \ref{bounded}.

\begin{proof}[Proof of Theorem \ref{bounded}]
By Proposition \ref{cob}, to prove the theorem it suffices to show that $Tr'$ is a faithful, positive state on $W$ and that $W$ acts by bounded operators on the GNS Hilbert space for $Tr'$.  That $Tr'$ is positive and faithful is one of the defining axioms of the subfactor planar algebra $\mc P$.  The fact that $W$ acts by bounded operators on the associated GNS Hilbert space may be proved along the same lines as \cite[Theorem 3.3]{jsw}.
\end{proof}

\section{The symmetric enveloping algebra}\label{sec:enveloping}

For $k \geq 0$ and $\epsilon = \pm$, let $p_{k,\epsilon} \in V_{2k,2k}^\epsilon(0,0)$ be the diagram consisting of $2k$ parallel strings, which is a projection in $V$.  Define $Gr_k^\epsilon(\mc P) \boxtimes Gr_k^\epsilon(\mc P)$ to be the compression $p_{k,\epsilon}Vp_{k,\epsilon}$, i.e.
\begin{equation*}
Gr_k^\epsilon(\mc P) \boxtimes Gr_k^\epsilon(\mc P) = \bigoplus_{k,s,t \geq 0} V_{2k,2k}^\epsilon(2s,2t).
\end{equation*}
Define $\tau_k \boxtimes \tau_k$ on $Gr_k^\epsilon(\mc P) \boxtimes Gr_k^\epsilon(\mc P)$ to be $\delta^{-2k}$ times the restriction of $Tr$.  

It follows from Theorem \ref{bounded} that $\tau_k \boxtimes \tau_k$ is a faithful, tracial state on $Gr_k^\epsilon(\mc P) \boxtimes Gr_k^\epsilon(\mc P)$.  Let $M_k^\epsilon \boxtimes M_k^\epsilon$ denote its GNS completion, which is naturally identified with the compression $p_{k,\epsilon}\mathfrak M p_{k,\epsilon}$.

As an example, consider the element $|\;| \in V_{0,0}^+(2,2) \subset M_0 \boxtimes M_0$ which consists of 2 vertical lines.  We will show that the moments of $|\;|$ are related to the famous meander problem (see e.g. \cite{dgg}).  Recall that a (generalized) \textit{meander} of order $n$ is a planar configuration of closed, non-intersecting loops crossing an infinite oriented line through $2n$ points.  Let $M_n^{(k)}$ denote the number of meanders of order $n$ with $k$ connected components, and define the \textit{meander polynomials} $m_n(q)$ by
\begin{equation*}
 m_n(q) = \sum_{k = 0}^n M_n^{(k)}q^k.
\end{equation*}
\begin{proposition}
If $\mc P$ is a subfactor planar algebra of modulus $\delta$, then
\begin{equation*}
 \tau_0 \boxtimes \tau_0\bigl((|\;|)^n\bigr) = m_n(\delta).
\end{equation*}

\end{proposition}

\begin{proof}
We have
\begin{equation*}
 \begin{tikzpicture}
  \draw[Box] (0,-.25) rectangle (2,-.75); \node[scale=.8] at (1,-.5) {$\sum TL$};
  \draw[Box] (0,.25) rectangle (2,.75); \node[scale=.8] at (1,.5) {$\sum TL$};
  \draw[verythickline](1,-.25) -- node[rcount,scale=.7] {$2n$} (1,.25);
  \node[left] at (-.4,0) {$\tau_0 \boxtimes \tau_0\bigl((|\;|)^n\bigr) = $};
 \end{tikzpicture}
\end{equation*}
As explained in \cite[Section 3.3]{dgg}, this is precisely $m_n(\delta)$.
\end{proof}

We will now show that $M_k \boxtimes M_k$ is naturally identified with Popa's symmetric enveloping algebra $M_k \boxtimes_{e_{k-1}} M_k^{op}$.  First note that there is a natural anti-automorphism $y \mapsto y^{op}$ of $Gr_k^\epsilon(\mc P) \boxtimes Gr_k^\epsilon(\mc P)$, determined by
\begin{equation*}
\begin{tikzpicture}[thick,scale=.75]
\draw[Box](0,0) rectangle (2,1);
  \draw[verythickline](0,.5) -- (-.5,.5); \draw[verythickline](2,.5) -- (2.5,.5);
 \draw[verythickline](1,0) -- (1,-.25); \draw[verythickline](1,1) -- (1,1.25);
 \node at (1,.5) {$y^{op}$}; \node[marked] at (-.1,1.05) {};
\node at (3,.5) {$=$};
\begin{scope}[xshift=4cm]
\draw[Box](0,0) rectangle (2,1);
  \draw[verythickline](0,.5) -- (-.5,.5); \draw[verythickline](2,.5) -- (2.5,.5);
 \draw[verythickline](1,0) -- (1,-.25); \draw[verythickline](1,1) -- (1,1.25);
 \draw[Box] (-.5,-.25) rectangle (2.5,1.25); \node[marked] at (-.6,1.35) {};
 \node[flip] at (1,.5) {$y$}; \node[marked] at (2.1,-.05) {};
\end{scope}
\end{tikzpicture} 
\end{equation*}
Since $y \mapsto y^{op}$ preserves the trace $\tau_k \boxtimes \tau_k$, this extends to an anti-automorphism of $M_k^{\epsilon} \boxtimes M_k^{\epsilon}$. 

Observe also that $Gr_k^\epsilon(\mc P) \boxtimes Gr_k^\epsilon(\mc P)$ contains a copy of $Gr_k^\epsilon(\mc P) \otimes Gr_k^\epsilon(\mc P)^{op}$ as follows:
\begin{equation*}
 \begin{tikzpicture}[scale=.75]
 \draw[Box] (0,0) rectangle (2,1); \node at (1,.5) {$x$};
 \draw[verythickline] (-.5,.5) -- (0,.5); \draw[verythickline] (2,.5) -- (2.5,.5);
 \draw[verythickline] (1,1) -- (1,1.25); \node[marked] at (-.1,1.05) {};
\begin{scope}[rotate around={180:(1,.5)}, yshift=1.5cm]
 \draw[Box] (0,0) rectangle (2,1); \node at (1,.5) {$y$}; \node[marked] at (-.1,1.05) {};
 \draw[verythickline] (-.5,.5) -- (0,.5); \draw[verythickline] (2,.5) -- (2.5,.5);
 \draw[verythickline] (1,1) -- (1,1.25);
\end{scope}
\draw[Box] (-.5,-1.75) rectangle (2.5,1.25); \node[marked] at (-.6,1.35) {};
\node[left=2mm,scale=1.1] at (-.5,-.25) {$x \otimes y^{op}  \mapsto$};
 \end{tikzpicture}
\end{equation*}
Since $\tau_k \boxtimes \tau_k$ restricts to the Voiculescu trace on $Gr_k^\epsilon(\mc P) \otimes Gr_k^\epsilon(\mc P)^{op}$, this inclusion extends to $M_k^\epsilon \otimes (M_k^\epsilon)^{op} \subset M_k^\epsilon \boxtimes M_k^\epsilon$.  Finally, if $k \geq 1$ and $\mathbf{f}_{k,\epsilon} \in M_k^\epsilon \boxtimes M_k^\epsilon$ is the projection
\begin{equation*}
 \begin{tikzpicture}[xscale=.5,yscale=.75]
  \draw[Box] (0,0) rectangle (2,1.5); \draw[thick] (0,.9) arc (90:-90:.3cm and .15cm);
  \draw[thick] (2,.9) arc(90:270:.3cm and .15cm); \draw[thickline] (0,1.2) -- node[scale=.7,count,rectangle, rounded corners] {$k-1$} (2,1.2);
  \draw[thickline] (0,.3) -- (2,.3);
  \node[left =1mm] at (0,.75) {$\mathbf f_{k,\epsilon} = \delta^{-1} \cdot$};
 \end{tikzpicture}
\end{equation*}
then $\mathbf{f}_{k,+}$ implements the Jones projections for both $M_{k-1} \subset M_k$ and $M_{k-1}^{op} \subset M_k^{op}$.  It now follows from \cite[Theorem 1.2]{pop1} and the discussion above that $W^*(M_k \otimes M_k^{op},\mathbf{f}_k)$ is isomorphic to $M_k \boxtimes_{e_{k-1}} M_k^{op}$, in particular it is a factor.  We will now show that $M_k \boxtimes M_k = W^*(M_k \otimes M_k^{op},\mathbf{f}_k)$.  First we need a lemma.

Let $k \geq 0$ and $\epsilon = \pm$, and define $c_{k,\epsilon} \in V_{2k,2k+2}^\epsilon(1,1)$ by
\begin{equation*}
 \begin{tikzpicture}[yscale=.5,xscale=.75]
 \draw[Box] (0,0) rectangle (1,2); \draw[verythickline] (0,1) -- node[count,rectangle,rounded corners,scale=.75] {$2k$} (1,1);
\draw[thick] (.5,2) arc (180:270:.5cm); \draw[thick] (.5,0) arc (180:90:.5cm);
\node[left=1mm] at (0,1) {$c_{k,\epsilon} =$};
 \end{tikzpicture}
\end{equation*}

It follows from \cite[Lemma 2]{gjs1} that $c_{k,\epsilon}c_{k,\epsilon}^*$ is invertible in $M_k^\epsilon \boxtimes M_k^\epsilon$.  Let $q_{k,\epsilon}$ be the initial projection of $c_{k,\epsilon}^*c_{k,\epsilon}$, 
\begin{equation*}
q_{k,\epsilon} = c_{k,\epsilon}^*(c_{k,\epsilon}c_{k,\epsilon}^*)^{-1}c_{k,\epsilon}.
\end{equation*}

\begin{lemma}\label{compress}
Fix $k \geq 0$ and $\epsilon = \pm$.
\begin{enumerate}
\item The map 
\begin{equation*}
 \theta_{k,\epsilon}(x) = c_{k,\epsilon}^*(c_{k,\epsilon}c_{k,\epsilon}^*)^{-1/2}x(c_{k,\epsilon}c_{k,\epsilon}^*)^{-1/2}c_{k,\epsilon}
\end{equation*}
 is an isomorphism of $M_k^{\epsilon} \boxtimes M_k^{\epsilon}$ onto $q_{k,\epsilon}(M_{k+1}^{-\epsilon} \boxtimes M_{k+1}^{-\epsilon})q_{k,\epsilon}$.
\item If $k \geq 1$, then $\theta_{k,\epsilon}$ restricts to an isomorphism
\begin{equation*}
W^*(M_k^{\epsilon} \otimes (M_k^{\epsilon})^{op},\mathbf{f}_{k,\epsilon}) \simeq q_{k,\epsilon}(W^*(M_{k+1}^{-\epsilon} \otimes  (M_{k+1}^{-\epsilon})^{op},\mathbf{f}_{k+1,-\epsilon}))q_{k,\epsilon}.
\end{equation*}
\end{enumerate}
\end{lemma}

\begin{proof}
(1) is straightforward, the inverse of $\theta_{k,\epsilon}$ is given by
\begin{equation*}
y \mapsto (c_{k,\epsilon}c_{k,\epsilon}^*)^{-1/2}c_{k,\epsilon}yc_{k,\epsilon}^*(c_{k,\epsilon}c_{k,\epsilon}^*)^{-1/2}.
\end{equation*}

For (2), since it is clear that $\theta_{k,\epsilon}$ takes $M_{k}^\epsilon \otimes (M_k^{\epsilon})^{op}$ onto $q_{k,\epsilon}(M_{k+1}^{-\epsilon} \otimes (M_{k+1}^{-\epsilon})^{op})q_{k,\epsilon}$ and $\mathbf f_{k,\epsilon}$ to $q_{k,\epsilon}\mathbf f_{k+1,-\epsilon}q_{k,\epsilon}$, it suffices to show that $q_{k,\epsilon}W^*(M_{k+1}^{-\epsilon} \otimes (M_{k+1}^{-\epsilon})^{op},\mathbf f_{k+1,-\epsilon})q_{k,\epsilon}$ is generated by $q_k(M_{k+1}^{-\epsilon} \otimes (M_{k+1}^{-\epsilon})^{op})q_{k,\epsilon}$ and $q_{k,\epsilon}\mathbf{f}_{k+1,-\epsilon}q_{k,\epsilon}$.
Note that $q_{k,\epsilon}$ is contained in the $II_1$ factor $M_1^\epsilon \otimes (M_1^\epsilon)^{op}$, so there are partial isometries $v_1,\dotsc,v_n \in M_1^\epsilon \otimes (M_1^\epsilon)^{op}$ such that $v_i^*v_i \leq q_{k,\epsilon}$ and $\sum_i v_i^*q_{k,\epsilon}v_i = 1$.  Let $m_1,\dotsc,m_l \in M_{k+1}^\epsilon \otimes (M_{k+1}^{-\epsilon})^{op}$, then
\begin{multline*}
 q_{k,\epsilon}(m_1\mathbf f_{k+1,-\epsilon}m_2\dotsb \mathbf f_{k+1,-\epsilon} m_l)q_{k,\epsilon} \\
= q_{k,\epsilon}m_1\bigl(\sum_{i_1=1}^n v_{i_1}^*q_{k,\epsilon}v_{i_1}\bigr)\mathbf f_{k+1,-\epsilon}m_2\bigl(\sum_{i_2=1}^n v_{i_2}^*q_{k,\epsilon}v_{i_2}\bigr)\dotsb m_{l-1}\bigl(\sum_{i_{l-1}=1}^n v_{i_{l-1}}^*q_{k,\epsilon}v_{i_{l-1}}\bigr)\mathbf f_{k+1,-\epsilon} m_l q_{k,\epsilon}\\
= \sum_{1 \leq i_1,\dotsc,i_{l-1} \leq n} q_{k,\epsilon}(m_1v_{i_1}^*)q_{k,\epsilon}\mathbf f_{k+1,-\epsilon} q_{k,\epsilon}(v_{i_1}m_2v_{i_2}^*)q_{k,\epsilon}\mathbf f_{k+1,-\epsilon} \dotsb \mathbf f_{k+1,-\epsilon}q_{k,\epsilon}(v_{i_{l-1}}m_l)q_{k,\epsilon},
\end{multline*}
where have used the fact that $\mathbf f_{k+1,-\epsilon}$ commutes with $M_1^{-\epsilon} \otimes (M_1^{-\epsilon})^{op}$.  The result follows.
\end{proof}

\begin{theorem}\label{envelope}
For any $k \geq 1$ we have
\begin{equation*}
 M_k \boxtimes M_k \simeq M_k \boxtimes_{e_{k-1}} M_k^{op}.
\end{equation*}
\end{theorem}

\begin{proof}
 As discussed above, we just need to show that $M_{k} \boxtimes M_k = W^*(M_k \otimes M_k^{op},\mathbf{f}_k)$.  If $x \in Gr_k(\mc P) \boxtimes Gr_k(\mc P)$, then by Lemma \ref{compress} we have $x \in W^*(M_k \otimes M_k^{op},\mathbf f_k)$ if $c_{k,+}^*xc_{k,+} \in W^*(M_{k+1}^- \otimes (M_{k+1}^-)^{op},\mathbf f_{k+1,-})$.  Iterating, it is enough to show that for $l$ sufficiently large we have
\begin{equation*}
\begin{tikzpicture}[scale=.5]
  \draw[Box] (0,0) rectangle (2,1); \node at (1,.5) {$x$};
  \draw[thickline] (1,1) -- (1,2);
  \draw[thickline] (1,0) -- (1,-1);
  \draw[thickline] (0,.5) -- (-1,.5);
  \draw[thickline] (2,.5) -- (3,.5);
  \draw[thickline] (-1,1) .. node[rcount, near end,scale=.75]{$2l$} controls ++(.75,0) and ++(0,-.75)  .. (0,2);
  \draw[thickline] (-1,0) .. node[rcount,near end,scale=.75]{$2l$} controls ++(.75,0) and ++(0,.75) .. (0,-1);
  \draw[thickline] (2,2) .. node[rcount, near start,scale=.75]{$2l$} controls ++(0,-.75) and ++(-.75,0) .. (3,1);
  \draw[thickline] (2,-1) .. node[rcount,near start,scale=.75]{$2l$} controls ++(0,.75) and ++(-.75,0) .. (3,0);
 \node[right] at (3,.5) {$\in W^*(M_{k+2l} \otimes M_{k+2l}^{op},\mathbf f_{k+2l})$};
\end{tikzpicture}
\end{equation*}

Let $x \in V_{2k,2k}(2s,2t)$, then for $l = t+k$ we have
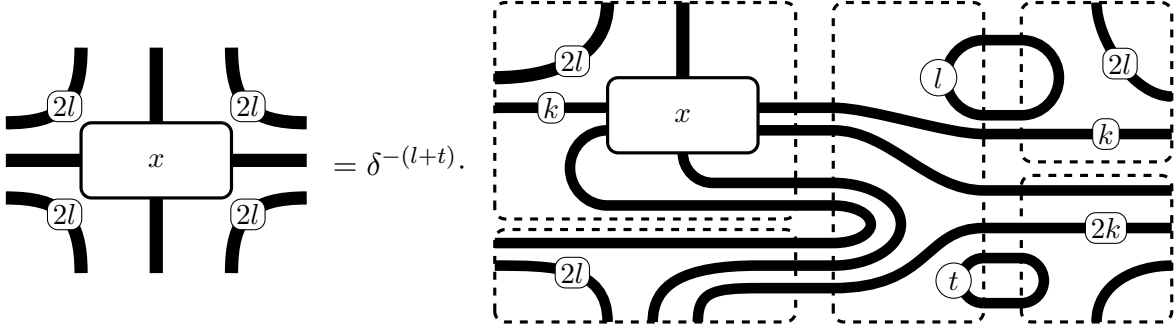
\begin{equation*}
 \begin{tikzpicture}
  \draw[Box] (0,0) rectangle (2,1); \node at (1,.5) {$x$};
  \draw[verythickline] (1,1) -- (1,2);
  \draw[verythickline] (1,0) -- (1,-1);
  \draw[verythickline] (0,.5) -- (-1,.5);
  \draw[verythickline] (2,.5) -- (3,.5);
  \draw[verythickline] (-1,1) .. node[rcount]{$2l$} controls ++(.75,0) and ++(0,-.75)  .. (0,2);
  \draw[verythickline] (-1,0) .. node[rcount]{$2l$} controls ++(.75,0) and ++(0,.75) .. (0,-1);
  \draw[verythickline] (2,2) .. node[rcount]{$2l$} controls ++(0,-.75) and ++(-.75,0) .. (3,1);
  \draw[verythickline] (2,-1) .. node[rcount]{$2l$} controls ++(0,.75) and ++(-.75,0) .. (3,0);
  \node[right = 2mm,scale=1.1] at (3,.5) {$= \delta^{-(l+t)}\cdot$};
\begin{scope}[xshift=7cm,yshift=.6cm]
\draw[Box] (0,0) rectangle (2,1); \node at (1,.5) {$x$};
  \draw[verythickline] (1,1) -- (1,2); 
  \draw[verythickline] (-1.5,1) .. node[rcount]{$2l$} controls ++(.75,0) and ++(0,-.75)  .. (0,2);
  \draw[thickline] (-1.5,-1.5) .. node[rcount]{$2l$} controls ++(.75,0) and ++(0,.75) .. (0,-2.25);
  \draw[thickline] (0,.6) -- node[rcount]{$k$} (-1.5,.6);
  \draw[thickline] (0,.3) arc (90:270:.5cm) -- (3,-.7) arc (90:-90:.5cm and .25cm) -- (-1.5,-1.2);
  \draw[thickline] (1,0) arc (180:270:.4cm) -- (3,-.4) arc (90:-90:.9cm and .55cm) -- (2,-1.5) .. controls ++(-1,0) and ++(0,.5) .. (.6,-2.25);
  \draw[thickline] (2,.6) -- (3,.6) .. controls ++(.5,0) and ++(-.5,0) .. (5,.25) -- node[pos=.65,rcount] {$k$} (7.5,.25);
  \draw[thickline] (2,.3) -- (3,.3) .. controls ++(.75,0) and ++(-.75,0) .. (5,-.5) -- (7.5,-.5);
  \draw[thickline] (7.5,-1) -- node[pos=.35,count,rectangle, rounded corners] {$2k$} (5,-1) .. controls ++(-.8,0) and ++(1.3,0) ..  (3,-1.8) -- (2,-1.8) .. controls ++(-.6,0) and ++(0,.4) .. (1.2,-2.25);
  \draw[thickline] (5,1.5) arc (90:270:.5cm) -- (5.5,.5) arc(-90:90:.5cm) -- (5,1.5); \node[count] at (4.4,1) {$l$};
  \draw[thickline] (5,-2) arc (270:90:.3cm) -- (5.5,-1.4) arc(90:-90:.3cm) -- (5,-2); \node[count] at (4.6,-1.7) {$t$};
  \draw[thickline] (6.5,2) .. node[rcount]{$2l$} controls ++(0,-.5) and ++(-.5,0) .. (7.5,.75);
  \draw[thickline] (6.5,-2.25) .. controls ++(0,.5) and ++(-.5,0) .. (7.5,-1.5);
  \draw[dashed,Box] (-1.5,-.88) rectangle (2.5,2);
  \draw[dashed,Box] (-1.5,-2.25) rectangle (2.5,-1.02);
  \draw[dashed,Box] (3,-2.25) rectangle (5,2);
  \draw[dashed,Box] (5.5,-2.25) rectangle (7.5,-.3);
  \draw[dashed,Box] (5.5,-.12) rectangle (7.5,2);
\end{scope}
 \end{tikzpicture}
\end{equation*}
is the product of three elements in $M_{k+2l} \boxtimes M_{k+2l}$.  The first and third terms are (manifestly) elements of $M_{k+2l} \otimes M_{k+2l}^{op}$.  The second term is a Temperley-Lieb element, but since $W^*(M_{k+2l} \otimes M_{k+2l}^{op},\mathbf{f}_{k+2l})$ contains all of the standard Temperley-Lieb generators
\begin{equation*}
\begin{tikzpicture}[scale=.75]
\draw[Box] (0,0) rectangle (1.5,1.5); \draw[thick] (0,.9) arc (90:-90:.2cm and .15cm);
  \draw[thick] (1.5,.9) arc(90:270:.2cm and .15cm); \draw[thickline] (0,1.2) -- node[rcount] {$i$} (1.5,1.2);
  \draw[thickline] (0,.3) -- (1.5,.3);
  \node[left =1mm] at (0,.75) {$e_i = $}; \node[right] at (1.5,.75) {$, \qquad (i = 0,\dotsc,2k+4l-2),$};
 \end{tikzpicture}
\end{equation*}
it contains any such Temperley-Lieb element.  This completes the proof.
\end{proof}

We will now give a diagrammatic computation of the index $[M_0 \boxtimes M_0:M_0 \otimes M_0^{op}]$.  Let $\Gamma$ be the principal graph of $\mc P$ (see e.g. \cite{ghj}), and let $\Gamma_+$ denote the collection of even vertices.  Associated to each $v \in \Gamma_+$ is an irreducible $M_0-M_0$ bimodule $X_v$.  Let $*$ denote the distinguished vertex, corresponding to $X_* = L^2(M_0)$, and for $k \geq 0$ let $\Gamma_{+,k}$ denote the set of vertices in $\Gamma$ whose distance from $*$ is less than or equal to $2k$.  Then we have
\begin{equation*}
 {}_{M_0}L^2(M_k)_{M_0} \simeq \bigoplus_{v \in \Gamma_{+,k}} X_v \otimes \mc H_v(k),
\end{equation*}
where $\mc H_v(k)$ are auxiliary finite-dimensional Hilbert spaces, whose dimensions we denote by $n_k(v)$.

It follows that
\begin{equation*}
{}_{M_0 \otimes M_0^{op}}L^2(M_k \otimes M_{k}^{op})_{M_0 \otimes M_0^{op}} \simeq \bigoplus_{v,w \in \Gamma_{+,k}} X_v \otimes \overline{X_w} \otimes (\mc H_v(k) \otimes \overline{\mc H_w(k)}),
\end{equation*}
where $\overline{X_w}$ is the contragredient bimodule and $\overline{\mc H_w(k)}$ is the conjugate Hilbert space.  Let $1_v(k) \in \mc H_v(k) \otimes \overline{\mc H_v(k)}$ be the unit under the natural identification with $\Hom_{\C}(\mc H_k(v))$, and define $p_k(v)$ to be the projection from $L^2(M_k \otimes M_k^{op})$ onto $X_v \otimes \overline{X_v} \otimes 1_v(k)$.  Then define
\begin{equation*}
 p_k = \sum_{v \in \Gamma_{+,k}} p_k(v).
\end{equation*}

Recall that $P_{2k}$ can be identified with $\Hom_{M_0,M_0}(L^2(M_{k}))$, in particular the central projections $q_k(v)$ of $P_{2k}$ are indexed by $v \in \Gamma_{+,k}$.  Let $t_k(v)$ denote the trace of a minimal projection in $P_{2k}(v) := P_{2k}\cdot q_{k}(v)$, the central component of $P_{2k}$ corresponding to $v \in \Gamma_{+,k}$.  If $\{e_{ij}(v):1 \leq i,j \leq n_k(v)\}$ is a choice of matrix units for $P_{2k}(v)$, then we have
\begin{equation*}
 p_k(v) = \frac{1}{n_k(v)} \sum_{1 \leq i,j \leq n_k(v)} e_{ij}(v) \otimes e_{ji}(v)^{op}.
\end{equation*}
Graphically, we will represent $p_k(v)$ using a Sweedler type convention as follows:
\begin{equation*}
\begin{tikzpicture}[scale=.75]
\draw[Box] (0,0) rectangle (2,1); \node[flip] at (1,.5) {$p^{(2)}_k(v)$};
\draw[Box] (0,1.5) rectangle (2,2.5); \node at (1,2) {$p^{(1)}_k(v)$};
\draw[medthick] (-.5,.5) -- (0,.5); 
\draw[medthick] (2,.5) -- (2.5,.5);
\draw[medthick] (-.5,2) -- (0,2); 
\draw[medthick] (2,2) -- (2.5,2);
\node[left] at (-.6,1.25) {$p_k(v) = $};
\end{tikzpicture}
\end{equation*}

\begin{lemma}\label{cleaver}
Let $x, y \in P_{2k}$ and $v \in \Gamma_{+,k}$, then we have
\begin{equation*}
\begin{tikzpicture}[scale=.7]
\draw[Box] (0,0) rectangle (2,1); \node[flip] at (1,.5) {$p^{(2)}_k(v)$};
\draw[Box] (0,1.5) rectangle (2,2.5); \node at (1,2) {$p^{(1)}_k(v)$};
\draw[Box] (2.5,0) rectangle (4.5,1); \node[flip] at (3.5,.5) {$y$};
\draw[Box] (2.5,1.5) rectangle (4.5,2.5); \node at (3.5,2) {$x$};
\draw[medthick] (4.5,.5) -- (4.75,.5) arc(90:-90:.5cm) -- (-.25,-.5) arc(270:90:.5cm) -- (0,.5); 
\draw[medthick] (4.5,2) -- (4.75,2) arc(-90:90:.5cm) -- (-.25,3) arc (90:270:.5cm) -- (0,2);
\draw[medthick] (2,.5) -- (2.5,.5);
\draw[medthick] (2,2) -- (2.5,2);
\begin{scope}[xshift=7.5cm]
\draw[Box] (2.5,0) rectangle (4.5,1); \node[flip] at (3.5,.5) {$y$};
\draw[Box] (2.5,1.5) rectangle (4.5,2.5); \node at (3.5,2) {$x$};
\draw[medthick] (4.5,2) -- (4.75,2) arc(90:0:.5cm) -- node[count,rectangle,rounded corners, scale=.8] {$q_k(v)$} (5.25,1) arc(0:-90:.5cm) -- (4.5,.5);\node[marked,scale=.8] at (6.05,1.25) {};
\draw[medthick] (2.5,2) -- (2.25,2) arc(90:180:.5cm) -- (1.75,1) arc(180:270:.5cm) -- (2.5,.5);
\node[left,scale=1.25] at (1.5,1.25) {$= \frac{\delta^{2k}\cdot t_k(v)}{n_k(v)} \cdot$};
\end{scope}
\end{tikzpicture}
\end{equation*}
\end{lemma}

\begin{proof}
Let $Tr_v$ denote the unnormalized trace on $P_{2k}(v)\simeq M_{n_k(v)}(\C)$, so that $\tau_k|_{P_{2k}(v)} = t_k(v) \cdot Tr_v$.  Then the left hand side of the formula in the statement is equal to
\begin{multline*}
 \frac{1}{n_k(v)}\sum_{1 \leq i,j \leq n_k(v)} \delta^{4k} \cdot t_{v}(k)^2 Tr_v(e_{ij}(v)x)Tr_v(ye_{ji}(v)) \\
= \frac{\delta^{2k}\cdot t_k(v)}{n_k(v)}\sum_{1 \leq i,j \leq n_k(v)} \delta^{2k} \cdot t_k(v) Tr_v(xyq_k(v)),
\end{multline*}
which is equal to the right hand side.
\end{proof}

Define $\Psi_k:Gr_k \otimes Gr_k^{op} \to Gr_0 \boxtimes Gr_0$ by linear extension of
\begin{equation*}
\begin{tikzpicture}[scale=.7]
\draw[Box] (0,0) rectangle (2,1); \node at (1,.5) {$y$}; \draw[thickline] (1,0) -- (1,-.25);
\draw[Box] (0,3) rectangle (2,4); \node at (1,3.5) {$x$};\draw [thickline] (1,4) -- (1,4.25);
\draw[Box] (0,1.5) rectangle (2,2.5); \node at (1,2) {$q_k(v)$}; \node[marked] at (2.2,2) {};
\draw[medthick] (2,.5) arc(-90:90:.375cm) -- (1.5,1.25) arc(270:180:.25cm);
\draw[medthick] (0,.5) arc(270:90:.375cm) -- (.5,1.25) arc(-90:0:.25cm);
\draw[medthick] (0,3.5) arc(90:270:.375cm) -- (.5,2.75) arc(90:0:.25cm);
\draw[medthick] (2,3.5) arc(90:-90:.375cm) -- (1.5,2.75) arc(90:180:.25cm);
\node[left] at (-1,2) {$\Psi_k(x \otimes y^{op}) = \displaystyle \sum_{v \in \Gamma_{+,k}} \sqrt{\frac{t_k(v)}{n_k(v)}}\; \cdot$};
\end{tikzpicture}
\end{equation*}

\begin{theorem}\label{pisom}
$\Psi_k$ extends to a $M_0 \otimes M_0^{op}$-bilinear partial isometry from $L^2(M_k \otimes M_k^{op})$ into $L^2(M_0 \boxtimes M_0)$, whose initial projection is equal to $p_k$.  
\end{theorem}

\begin{proof}
Let $x_1,x_2,y_1,y_2 \in Gr_k$, then $\langle \Psi_k(x_1 \otimes y_1^{op}), \Psi_k(x_2 \otimes y_2^{op})\rangle$ is equal to
\begin{equation*}
 \begin{tikzpicture}[scale=.6]
 \draw[Box](0,0) rectangle (2,1); \node at (1,.5) {$x_1$};
\draw[Box](0,-3) rectangle (2,-2); \node at (1,-2.5) {$y_1$};
\draw[Box] (-2.5,-1.5) rectangle (-.5,-.5); \node at (-1.5,-1) {$q_k(v)$}; \node[marked,scale=.8] at (-.3,-1) {};
\draw[medthick] (2,-2.5) arc(-90:90:.35cm) -- (-1,-1.8) arc(270:180:.3cm);
\draw[medthick] (0,-2.5) -- (-1.7,-2.5) arc(270:180:.3cm) -- (-2,-1.5);
\draw[medthick] (0,.5) -- (-1.7,.5) arc(90:180:.3cm) -- (-2,-.5);
\draw[medthick] (2,.5) arc(90:-90:.35cm) -- (-1,-.2) arc(90:180:.3cm);

\begin{scope}[xshift=6cm,xscale=-1]
 \draw[Box](0,0) rectangle (2,1); \node at (1,.5) {$x_2^*$};
\draw[Box](0,-3) rectangle (2,-2); \node at (1,-2.5) {$y_2^*$};
\draw[Box] (-2.5,-1.5) rectangle (-.5,-.5); \node at (-1.5,-1) {$q_w(k)$}; \node[marked,scale=.8] at (-.3,-1) {};
\draw[medthick] (2,-2.5) arc(-90:90:.35cm) -- (-1,-1.8) arc(270:180:.3cm);
\draw[medthick] (0,-2.5) -- (-1.7,-2.5) arc(270:180:.3cm) -- (-2,-1.5);
\draw[medthick] (0,.5) -- (-1.7,.5) arc(90:180:.3cm) -- (-2,-.5);
\draw[medthick] (2,.5) arc(90:-90:.35cm) -- (-1,-.2) arc(90:180:.3cm);
\end{scope}

\draw[medthick] (1,1) arc(180:90:.6cm) -- node[Box,draw,fill=white] {$\sum TL$} (4.4,1.6) arc(90:0:.6cm);
\draw[medthick] (1,-3) arc(180:270:.6cm) -- node[Box,draw,fill=white] {$\sum TL$} (4.4,-3.6) arc(-90:0:.6cm);

\draw[Box,dashed] (-.25,-4.4) rectangle (6.25,-1.25); \draw[Box,dashed] (-.25,-.75) rectangle (6.25,2.4);
\node[left=2mm] at (-2.5,-1) {$\displaystyle \sum_{v,w \in \Gamma_{+,k}} \sqrt{\frac{t_k(v)t_w(k)}{n_k(v)n_w(k)}}$};
 \end{tikzpicture}
\end{equation*}
Let $x,y^{op}$ denote the elements of $P_{2k},P_{2k}^{op}$ represented by the dashed boxes above.  Since $q_k(v)$ and $q_w(k)$ are central, the terms in the sum above for which $v \neq w$ are equal to zero.  But applying Lemma \ref{cleaver} to $x,y$, we see that the term corresponding to $v=w$ is equal to $\langle p_k(v)\cdot (x_1 \otimes y_1^{op}), x_2 \otimes y_2^{op}\rangle$.  So we have proved that
\begin{equation*}
 \langle \Psi_k(x_1 \otimes y_1^{op}), \Psi_k(x_2 \otimes y_2^{op})\rangle = \langle p_k\cdot (x_1 \otimes y_1^{op}), x_2 \otimes y_2^{op}\rangle,
\end{equation*}
from which the result follows.
\end{proof}

From this we can deduce the following corollary, which is due to Ocneanu \cite{ocn} in the finite-depth case, and Popa \cite{pop3} in the general case.

\begin{corollary} \label{bimod}
We have
\begin{equation*}
 {}_{M_0 \otimes M_0^{op}} L^2(M_0 \boxtimes M_0)_{M_0 \otimes M_0^{op}} \simeq \bigoplus_{v \in \Gamma_+} X_v \otimes \overline{X_v}.
\end{equation*}
\end{corollary}

\begin{proof}
Let $\Xi_k$ denote the range of $\Psi_k$, which is a closed subspace of $L^2(M_0 \boxtimes M_0)$ invariant under the left and right actions of $M_0 \otimes M_0^{op}$.  By Theorem \ref{pisom}, $\Xi_k$ is isomorphic to
\begin{equation*}
\bigoplus_{v \in \Gamma_{+,k}} X_v \otimes \overline{X_v}
\end{equation*}
as a $M_0 \otimes M_0^{op}$-bimodule.  It follows that $\overline{\bigcup_{k \geq 0} \Xi_k}$ is isomorphic as a $M_0 \otimes M_0^{op}$-bimodule to
\begin{equation*}
\bigoplus_{v \in \Gamma_+} X_v \otimes \overline{X_v}.
\end{equation*}
Now if $x \in V_{0,0}(2s,2t)$ then $x = \Psi_t(x_t)$, where $x_t \in M_{t} \otimes M_{t}^{op}$ is defined by
\begin{equation*}
\begin{tikzpicture}[scale=.6]
\draw[Box] (0,2.5) rectangle (2,3.5); \node at (1,3) {$x$};\draw [thickline] (1,3.5) -- (1,3.75);
\draw[Box] (0,0) rectangle (2,1); \node at (1,.5) {$q_t(v)$}; \node[marked] at (2.2,.5) {};
\draw[thickline](1,0) -- (1,-.25);
\draw[medthick] (1.25,2.5) arc(180:270:.25cm) -- (2.5,2.25);
\draw[medthick] (.75,2.5) arc(0:-90:.25cm) -- (-.5,2.25);
\draw[medthick] (1.25,1) arc(180:90:.25cm) -- (2.5,1.25);
\draw[medthick] (.75,1) arc(0:90:.25cm) -- (-.5,1.25);
\draw[Box,dashed](-.5,-.25) rectangle (2.5,1.6); \draw[Box,dashed] (-.5,1.9) rectangle (2.5,3.75);
\node[left=2mm,scale=1] at (-.75,1.75) {$x_t = \displaystyle\sum_{v \in \Gamma_{+,t}} \sqrt{\frac{n_k(v)}{t_k(v)}} \;\cdot $};
\end{tikzpicture}
\end{equation*}
So $\bigcup_{k \geq 0} \Xi_k$ contains $Gr_0 \boxtimes Gr_0$ and is therefore dense in $L^2(M_0 \boxtimes M_0)$, which completes the proof. 
\end{proof}

\begin{remark}
Note that it follows from the proof of Corollary \ref{bimod} that $\Psi_k$ is surjective if the depth of $\mc P$ is less than or equal to $2k$, i.e. if $\Gamma_+ = \Gamma_{+,k}$.  
\end{remark}

\begin{corollary} \label{index}
Let $k \geq 0$, then if $\Gamma_+$ is finite the index $[M_k \boxtimes M_k:M_k \otimes M_k^{op}]$ is equal to the global index
\begin{equation*}
 I = \sum_{v \in \Gamma_+} \dim_{M_0}(X_v)^2.
\end{equation*}
Otherwise $[M_k \boxtimes M_k:M_k \otimes M_k^{op}] = \infty$. 
\end{corollary}

\begin{proof}
For $k = 0$ this is clear from the previous corollary.  The result then follows for $k \geq 1$ from Lemma \ref{compress} and the fact that the global index of $\mc P^{op}$ (the planar algebra obtained from $\mc P$ by reversing shading) is the same as that of $\mc P$.
\end{proof}

\section{Derivations on planar algebra factors}\label{sec:derivations}

Throughout this section we will denote $A = Gr_0(\mc P)$ and $M = M_0(\mc P)$, where $\mc P$ is a subfactor planar algebra.
Define
\begin{equation*}
 \Phi_o = \bigoplus_{s,t \geq 0} V_{1,0}^-(2s+1,2t), \qquad \Phi_e = \bigoplus_{s,t \geq 0} V_{1,0}^+(2s,2t+1), \qquad \Phi = \Phi_o \oplus \Phi_e.
\end{equation*}
Observe that $\Phi$ has a natural right $A \otimes A^{op}$-module structure.

Given $Q \in \Phi_o$ and $x \in P_{n}$, define $\widetilde \delta_Q(x) \in A \boxtimes A$ by
\begin{equation*}
 \begin{tikzpicture}[scale=.75]
  \draw[Box](0,0) rectangle (2,1); \node at (1,.5) {$x$};
  \draw[Box](4,0) rectangle (6,1); \node at (5,.5) {$Q$};
  \draw[verythickline] (.5,1) -- (.5,1.75);
  \draw[verythickline] (1.5,1) arc(180:0:.5cm) -- node[count,rectangle,rounded corners] {$k$} (2.5,-.5);
  \draw[verythickline] (5,1) -- (5,1.75); \draw[verythickline] (5,0) -- (5,-.5);
  \draw[thick] (1,1) arc(180:90:.75cm) -- (2.5,1.75) arc(90:0:.75cm) arc(180:270:.5cm) -- (4,.5);
  \node[left] at (-.25,.25) {$\widetilde \delta_Q(x) = \displaystyle \sum_{\substack{0 \leq k \leq 2n\\ k\text{ even}}}$};
 \end{tikzpicture}
\end{equation*}
Extend $\widetilde \delta_Q$ linearly to a obtain a map $\widetilde \delta_Q:A \to A \boxtimes A$.  If $Q \in \Phi_e$ we define $\widetilde \delta_Q:A \to A \boxtimes A$ likewise, except that we sum over odd instead of even $k$ in the formula above.  We then extend $Q \mapsto \widetilde \delta_Q$ linearly to $\Phi$.

\begin{proposition}If $Q \in \Phi$, then $\widetilde\delta_Q$ is a derivation from $A$ into $A \boxtimes A$.  The assignment $Q \mapsto \widetilde \delta_Q$ is a $A \otimes A^{op}$-linear map from $\Phi$ into $\Der(A,A \boxtimes A)$.
\end{proposition}

\begin{proof}
For the first statement, it suffices to consider $Q \in \Phi_o$ or $Q \in \Phi_e$.  In either case the result is clear from drawing the relevant diagram.  The second statement is obvious.
\end{proof}

Let $E_{M \otimes M^{op}}$ denote the unique trace-preserving conditional expectation from $M \boxtimes M$ onto $M \otimes M^{op}$, which is given by the orthogonal projection $L^2(M \boxtimes M) \to L^2(M \otimes M^{op})$.  We will now show that this projection may also be computed with respect to the orthogonal basis of Section \ref{sec:orthogonal}.   Let $\langle \langle \cdot, \cdot \rangle\rangle$ denote the inner product on $A \boxtimes A$ induced by $\delta^{-2k}$ times the restriction of $Tr'$.   

\begin{lemma}\label{oproj}
Let $P$ denote the orthogonal projection of $A \boxtimes A$ onto $A \otimes A^{op}$ with respect to $\langle\langle \cdot , \cdot \rangle \rangle$.  If $Q \in A \boxtimes A$ then
\begin{equation*}
 E_{M \otimes M^{op}}(Q) = P(Q).
\end{equation*}
 
\end{lemma}

\begin{proof}
If $Q \in A \boxtimes A$ and $a \otimes b \in A \otimes A^{op}$ we have 
\begin{equation*}
\begin{tikzpicture}[yscale=.75,xscale=.75]
\draw[Box] (0,0) rectangle (2,1); \node at (1,.5) {$Q$};
\draw[Box] (3,0) rectangle (5,.4); \node[scale=.8] at (4,.2) {$b^*$};
\draw[Box] (3,.6) rectangle (5,1); \node[scale=.8] at (4,.8) {$a^*$};
\draw[Box] (3,1.25) rectangle (5,1.75); \node[scale=.75] at (4,1.5) {$\sum TL$}; \draw[verythickline] (4,1) -- (4,1.25);
\draw[Box] (3,-.25) rectangle (5,-.75); \node[scale=.75] at (4,-.5) {$\sum TL$}; \draw[verythickline] (4,0) -- (4,-.25);
\draw[thickline] (1,1) -- (1,1.75) arc(180:90:.25cm) -- (3.75,2) arc (90:0:.25cm);
\draw[thickline] (1,0) -- (1,-.75) arc(180:270:.25cm) -- (3.75,-1) arc (270:360:.25cm);
\node[left] at (0,.5) {$\langle Q, a \otimes b \rangle = $};
\begin{scope}[xshift=6cm]
\draw[Box] (0,0) rectangle (2,1); \node at (1,.5) {$P(Q)$};
\draw[Box] (3,0) rectangle (5,.4); \node[scale=.8] at (4,.2) {$b^*$};
\draw[Box] (3,.6) rectangle (5,1); \node[scale=.8] at (4,.8) {$a^*$};
\draw[Box] (3,1.25) rectangle (5,1.75); \node[scale=.75] at (4,1.5) {$\sum TL$}; \draw[verythickline] (4,1) -- (4,1.25);
\draw[Box] (3,-.25) rectangle (5,-.75); \node[scale=.75] at (4,-.5) {$\sum TL$}; \draw[verythickline] (4,0) -- (4,-.25);
\draw[thickline] (1,1) -- (1,1.75) arc(180:90:.25cm) -- (3.75,2) arc (90:0:.25cm);
\draw[thickline] (1,0) -- (1,-.75) arc(180:270:.25cm) -- (3.75,-1) arc (270:360:.25cm);
\node[left=1mm] at (0,.5) {$=$}; \node[right] at (5,.5) {$= \langle P(Q), a \otimes b \rangle,$};
\end{scope}
\end{tikzpicture}
\end{equation*}
so that $P(Q)$ is equal to the orthogonal projection of $Q$ onto $M \otimes M^{op}$ in $L^2(\tau_k \boxtimes \tau_k)$.
\end{proof}

\begin{corollary}\hfill\label{proj}
\begin{enumerate}
 \item $E_{M \otimes M^{op}}(V_{0,0}^+(2s,2t)) \subset P_{s} \otimes P_t$, in particular $E_{M \otimes M^{op}}(A \boxtimes A) \subset A \otimes A^{op}$.
\item Suppose that $Q \in A \boxtimes A \cap (A \otimes A^{op})^\perp$.  Write 
\begin{equation*}
Q = \sum_{s,t \geq 0} Q_{s,t},
\end{equation*}
where $Q_{s,t} \in V^+_{0,0}(2s,2t)$ are nonzero for finitely many $s,t \geq 0$.  Then $Q_{s,t} \in (A \otimes A^{op})^\perp$ for all $s,t \geq 0$.
\item If $Q \in V^+_{0,0}(2s,2t) \cap (A \otimes A^{op})^\perp$, then
\begin{equation*}
 \begin{tikzpicture}[scale=.5]
 \draw[Box] (0,0) rectangle (2,1); \node at (1,.5) {$Q$};
 \draw[verythickline] (1,0) -- (1,-.25); \draw[medthick] (.5,1) arc (180:0:.5cm);
 \node[right] at (2.2,.5) {$ = 0.$};
 \end{tikzpicture}
\end{equation*}
\end{enumerate}
\end{corollary}

\begin{proof}
(1) and (2) are immediate from the previous lemma.  Let $Q$ be as in (3), then if $x \in P_{t}$ we have
\begin{equation*}
 \begin{tikzpicture}[yscale=.75,xscale=.5]
 \draw[Box] (0,0) rectangle (2,1); \node at (1,.5) {$Q$};
 \draw[medthick] (.5,1) arc (180:0:.5cm and .25cm);
 \draw[Box] (3,0) rectangle (5,.4); \node[flip] at (4,.2) {$x$};
 \draw[thickline] (1,0) arc(180:270:.25cm) -- (3.75,-.25) arc (-90:0:.25cm);
 \node[right] at (5.25,.5) {$=$};
\begin{scope}[xshift=6.5cm]
  \draw[Box] (0,0) rectangle (2,1); \node at (1,.5) {$P(Q)$};
 \draw[Box] (3,0) rectangle (5,.4); \node[flip] at (4,.2) {$x$};
 \draw[medthick] (1,0) arc(180:270:.25cm) -- (3.75,-.25) arc (-90:0:.25cm);
 \draw[medthick] (.5,1) arc(180:90:.5cm) -- (4.25,1.5) arc (90:0:.25cm) --(4.5,1) arc(0:-180:.5cm and .25cm) arc (0:90:.25cm) -- (1.75,1.25) arc (90:180:.25cm);
 \draw[Box,dashed] (3,.6) rectangle (5,1); \node[right] at (5,.5) {$ = 0,$};
\end{scope}
 \end{tikzpicture}
\end{equation*}
from which the result follows.
\end{proof}

Given $Q \in \Phi$, define $\delta_Q = E_{M \otimes M^{op}} \circ \widetilde \delta_Q$.  Note that the range of $\delta_Q$ is contained in $A \otimes A^{op}$ by the previous corollary.  Since $E_{M \otimes M^{op}}$ commutes with the left and right actions of $A \otimes A^{op}$, it is clear that $\delta_{Q} \in \Der(A,A \otimes A^{op})$ and that the assignment $Q \mapsto \delta_Q$ is $A \otimes A^{op}$-linear.  

We will now determine the kernel of the mapping $Q \mapsto \delta_Q$.  Define 
\begin{equation*}
\Omega_o = \bigoplus_{s,t \geq 0} V_{0,0}^- (2s+1,2t+1), \qquad \Omega_e = \bigoplus_{s,t \geq 0} V_{0,0}^+(2s,2t), \qquad \Omega = \Omega_o \oplus \Omega_e.
\end{equation*}
Observe that $\Omega_e = A \boxtimes A$, in particular $A \otimes A^{op} \subset \Omega$.  The orthogonal complement of $A \otimes A^{op}$ in $\Omega$ is $\Omega_o \oplus (A \otimes A^{op})^\perp \cap A \boxtimes A$.

Given $R \in \Omega_o$, define $\rho(R) \in \Phi$ by
\begin{equation*}
\begin{tikzpicture}
\draw[Box] (0,0) rectangle (1.5,1); \node at (.75,.5) {$R$};
\draw[verythickline] (.75,0) -- (.75,-.25); \draw[verythickline] (.75,1) -- (.75,1.25);
\draw[thick] (-.75,.5) -- (-.5,.5) arc (-90:0:.25cm) -- (-.25,1.25);
\node at (2.25,.5) {$-$};
\draw[Box] (3.5,0) rectangle (5,1); \node at (4.25,.5) {$R$};
\draw[thick] (2.75,.5) -- (3,.5) arc (90:0:.25cm) -- (3.25,-.25);
\draw[verythickline] (4.25,0) -- (4.25,-.25); \draw[verythickline] (4.25,1) -- (4.25,1.25);
\node[left] at (-1,.5) {$\rho(R) = $};
\end{tikzpicture}
\end{equation*}
Likewise if $R \in \Omega_e$ we define $\rho(R) \in \Phi$ by the same formula, but note that the shading of the ``corners'' must be reversed.  We then extend $\rho$ to $\Omega$ linearly.

\begin{proposition}
The restriction of $\rho$ to $\Omega \cap (A \otimes A^{op})^\perp$ is an $A \otimes A^{op}$-linear isomorphism onto the kernel of the mapping $Q \mapsto \delta_Q$.  
\end{proposition}

\begin{proof}
If $R \in \Omega_o$, then it is immediate from the defining formula that $\widetilde \delta_{\rho(R)} = 0$.  On the other hand, if $R \in \Omega_e$ then it is easy to see that $\widetilde \delta_{\rho(R)}(x) = [x,R]$, and therefore $\delta_{\rho(R)}(x) = [x,E_{M \otimes M^{op}}(R)]$.  So we have $\delta_{\rho(R)} = 0$ if $R \perp A \otimes A^{op}$.  In fact since $M$ is diffuse, any Hilbert-Schmidt operator commuting with $A$ must be zero, and so we have $\delta_{\rho(R)} = 0$ if and only if $R \perp A \otimes A^{op}$.

It remains to show that if $Q \in \Phi$ is such that $\delta_Q = 0$, then $Q = \rho(R)$ for a unique $R \in \Omega$.  Write
\begin{equation*}
 Q = \sum_{s,t \geq 0} Q_{s,t}
\end{equation*}
where $Q_{s,t} \in V_{1,0}^{(-1)^s}(s,t)$ are nonzero for finitely many $s,t$.  For $s,t \geq 1$ with $s + t$ even, define $R_{s,t} \in V_{0,0}^{(-1)^s}(s,t)$ by
\begin{equation*}
\begin{tikzpicture}[yscale=.75]
 \draw[Box] (0,0) rectangle (2,1); \node at (1,.5) {$Q_{s+k,t+1-k}$};
 \draw[thickline] (.25,1) arc (180:0:.5cm); \node[draw,fill=white,scale=.75,circle] at (.75,1.5) {$k$};
 \draw[thickline] (1.75,1) -- (1.75,1.75);
 \draw[thick] (0,.5) arc (270:180:.25cm) -- (-.25,1.75);
 \draw[verythickline] (-.75,1.75) -- (-.75,-.5); \node[draw,fill=white,rectangle, rounded corners,scale=.75] at (-.75,.3) {$k-1$};
 \draw[verythickline] (1,0) -- (1,-.5); 
 \node[left] at (-1.2,.6) {$R_{s,t} = \displaystyle\sum_{k=1}^{\min\{s,t+1\}}$};
\end{tikzpicture}
\end{equation*}
Note that $R_{s,t} = 0$ for all but finitely many $s,t$, and set $R = \sum_{s,t} R_{s,t} \in \Omega$.  We claim that $Q = \rho(R)$, note that it will then follow from the discussion above that $R \in (A \otimes A^{op})^\perp$.

We need to show that for any $s,t \geq 0$ we have
\begin{equation*}
\begin{tikzpicture}
\draw[Box] (0,0) rectangle (1.5,1); \node at (.75,.5) {$R_{s-1,t}$};
\draw[verythickline] (.75,0) -- (.75,-.25); \draw[verythickline] (.75,1) -- (.75,1.25);
\draw[thick] (-.75,.5) -- (-.5,.5) arc (-90:0:.25cm) -- (-.25,1.25);
\node at (2.25,.5) {$-$};
\draw[Box] (3.5,0) rectangle (5,1); \node at (4.25,.5) {$R_{s,t-1}$};
\draw[thick] (2.75,.5) -- (3,.5) arc (90:0:.25cm) -- (3.25,-.25);
\draw[verythickline] (4.25,0) -- (4.25,-.25); \draw[verythickline] (4.25,1) -- (4.25,1.25);
\node[left] at (-1,.5) {$Q_{s,t} = $};
\end{tikzpicture}
\end{equation*}
We will prove this by induction on $s$.  First suppose that $s = 0$, we claim that $Q_{s,t} = 0$ for any $t \geq 0$.  Indeed, this follows from $\delta_Q(\cup) = 0$ and (2) of Corollary \ref{proj}.

So assume that $s \geq 1$.  Let $x$ be the following Temperley-Lieb element with $2(s+1)$ boundary points:
\begin{equation*}
\begin{tikzpicture}
\draw[Box] (0,0) rectangle (1,.5);
\draw[thickline] (.25,.5) arc(180:360:.25cm and .3cm); \node[left=2mm] at (0,.25) {$x = $};
\end{tikzpicture}
\end{equation*}
Since $\delta_Q(x) = 0$, applying (2) of Corollary \ref{proj} we find that
\begin{equation*}
 \begin{tikzpicture}
  \draw[Box] (0,0) rectangle (2,1); \node[scale=1] at (1,.5) {$Q_{s-k,t+k}$};
  \draw[verythickline] (1,1) -- (1,1.25); \draw[verythickline] (1,0) -- (1,-.25);
  \draw[thickline] (-1.25,1.25) arc (180:360:.5cm); \node[draw,fill=white,count,rectangle,rounded corners,scale=.7] at (-.75,.75) {$k-1$};
  \draw[verythickline] (-1.75,1.25) -- (-1.75,-.25); 
  \draw[thick] (0,.5) -- (-1.25,.5) arc(270:180:.25cm) -- (-1.5,1.25);
   \node[left] at (-2.25,.5) {$\displaystyle\sum_{k=1}^{s}$};
\begin{scope}[xshift=6.25cm]
\draw[Box] (0,0) rectangle (2,1); \node[scale=1] at (1,.5) {$Q_{s+k,t-k}$};
  \draw[verythickline] (1,1) -- (1,1.25); \draw[verythickline] (1,0) -- (1,-.25);
  \draw[thickline] (-1.25,-.25) arc (180:0:.5cm); \node[draw,fill=white,count,rectangle,rounded corners,scale=.8] at (-.75,.25) {$k$};
  \draw[verythickline] (-1.75,1.25) -- (-1.75,-.25);
  \draw[thick] (0,.5) -- (-1.25,.5) arc(90:180:.25cm) -- (-1.5,-.25);
 \node[left] at (-2.25,.5) {$+ \displaystyle\sum_{k=0}^{\min\{s,t\}}$};
 \node[right] at (2,.5) {$\in (A \otimes A^{op})^\perp$,};
\end{scope}
 \end{tikzpicture}
\end{equation*}
where the number of parallel lines appearing at left is determined by the condition that each term is in $V_{0,0}^+(2s,t+s+1)$.

Applying induction to the first sum and simplifying, we are left with
\begin{equation*}
\begin{tikzpicture}
\draw[Box] (0,0) rectangle (2,1); \node[scale=1] at (1,.5) {$Q_{s+k,t-k}$};
  \draw[verythickline] (1,1) -- (1,1.25); \draw[verythickline] (1,0) -- (1,-.25);
  \draw[thickline] (-1.25,-.25) arc (180:0:.5cm); \node[draw,fill=white,count,rectangle,rounded corners,scale=.8] at (-.75,.25) {$k$};
  \draw[verythickline] (-1.75,1.25) -- (-1.75,-.25);
  \draw[thick] (0,.5) -- (-1.25,.5) arc(90:180:.25cm) -- (-1.5,-.25);
 \node[left] at (-2.25,.5) {$\displaystyle\sum_{k=0}^{\min\{s,t\}}$};
 \begin{scope}[xshift=4.5cm]
  \draw[Box] (0,0) rectangle (2,1); \node at (1,.5) {$R_{s-1,t}$};
  \draw[verythickline] (1,1) -- (1,1.25); \draw[verythickline] (1,0) -- (1,-.25);
 \draw[verythickline] (-.75,1.25) -- node[draw,fill=white,count,rectangle,rounded corners,scale=.8] {$s+1$} (-.75,-.25);
 \node[left] at (-1.5,.5) {$-$}; 
\node[right] at (2,.5) {$\in (A \otimes A^{op})^\perp.$};
 \end{scope}
\end{tikzpicture}
\end{equation*}
By (3) of Corollary \ref{proj}, we have
\begin{equation*}
\begin{tikzpicture}
\begin{scope}[xshift=-6.5cm]
  \draw[Box] (0,0) rectangle (2,1); \node at (1,.5) {$Q_{s,t}$};
  \draw[verythickline] (1,0) -- (1,-.25);
 \draw[verythickline] (-1,-.25) -- (-1,1) arc(180:90:.25cm) -- (.75,1.25) arc (90:0:.25cm); 
\draw[thick] (-.5,-.25) arc (180:90:.5cm and .75cm);
\node[right] at (2.25,.5) {$+$};
\end{scope}
\draw[Box] (0,0) rectangle (2,1); \node[scale=1] at (1,.5) {$Q_{s+k,t-k}$};
  \draw[verythickline] (1,0) -- (1,-.25);
  \draw[thickline] (-1.25,-.25) arc (180:0:.5cm); \node[draw,fill=white,count,rectangle,rounded corners,scale=.8] at (-.75,.25) {$k$};
  \draw[verythickline] (-1.75,-.25) -- (-1.75,1.5) arc (180:90:.25cm) -- (1.35,1.75) arc (90:0:.25cm) -- (1.6,1);
  \draw[thickline] (.4,1) arc(180:0:.4cm); \node[draw,fill=white,count,rectangle,rounded corners,scale=.9] at (.8,1.4) {$k$};
  \draw[thick] (0,.5) -- (-1.25,.5) arc(90:180:.25cm) -- (-1.5,-.25);
 \node[left] at (-2.25,.5) {$\displaystyle\sum_{k=1}^{\min\{s,t\}}$};
 \begin{scope}[xshift=4.5cm]
  \draw[Box] (0,0) rectangle (2,1); \node at (1,.5) {$R_{s-1,t}$};
  \draw[verythickline] (1,0) -- (1,-.25);
 \draw[verythickline] (-1,-.25) -- node[draw,fill=white,count,rectangle,rounded corners,scale=.8] {$s-1$} (-1,1) arc(180:90:.25cm) -- (.75,1.25) arc (90:0:.25cm); \draw[thick] (-.75,-.25) arc (180:0:.35cm);
 \node[left] at (-1.5,.5) {$-$}; 
\node[right] at (2,.5) {$= 0.$};
 \end{scope}
\end{tikzpicture}
\end{equation*}
After rotating $s$ strings clockwise from bottom to top, and one more from the bottom to the left, we obtain the desired result.

Finally, if $Q = \rho(R')$ and $R$ is defined in terms of $Q$ as above, it is straightforward to check that $R = R'$, which proves that $\rho$ is injective.  Since $\rho$ is clearly linear over $A \otimes A^{op}$, this completes the proof.
\end{proof}

We now compute the ``conjugate variable'' $\delta_Q^*(1 \otimes 1)$.  This is closely related to the Schwinger-Dyson equation from \cite[Lemma 12]{gjsz}.

\begin{proposition}
Let $Q \in V_{1,0}^{(-1)^s}(s,t)$, and view $\delta_Q$ as a densely defined unbounded operator $\delta_Q:L^2(M) \to L^2(M \otimes M^{op})$.  Then $1 \otimes 1 \in \mathfrak D(\delta_Q^*)$, and
\begin{equation*}
 \begin{tikzpicture}[xscale=.75,yscale=.6]
\draw[Box] (0,0) rectangle (2,1); \node at (1,.5) {$Q$};
\draw[verythickline] (1,1) -- (1,2); 
\draw[thick] (0,.5) arc(270:180:.25cm) -- (-.25,2);
\draw[thickline] (1,0) -- (1,-.5) arc(0:-90:.25cm) -- (-.25,-.75) arc (270:180:.25cm) -- (-.5,2);
\node[left] at (-1,.5) {$(\delta_Q^*(1 \otimes 1))^* = $};
\begin{scope}[xshift=5.5cm]
\draw[Box] (0,0) rectangle (2.5,1); \node at (1.25,.5) {$Q$};
\draw[thick] (0,.5) arc(270:180:.25cm) -- (-.25,1.75) arc (180:90:.25cm) -- (1.25,2) arc(90:0:.25cm) -- (1.5,1);
\draw[Box] (.25,1.25) rectangle (1.25,1.75); \node[scale=.5] at (.75,1.5) {$\sum TL$}; 
\draw[thickline] (2,1) -- node[count,scale=.7] {$k$} (2,2); \draw[verythickline] (.75,1) -- (.75,1.25);
\draw[thickline] (1.25,0) -- (1.25,-.5) arc(0:-90:.25cm) -- (-.25,-.75) arc (270:180:.25cm) -- (-.5,2);
\node[left] at (-1,.25) {$- \displaystyle \sum_{\substack{0 \leq k \leq s\\ k +t \in 2\Z}}$};
\end{scope}
\begin{scope}[xshift=11.5cm,yshift=.5cm]
\draw[Box] (0,0) rectangle (2.5,1); \node at (1.25,.5) {$Q$};
\draw[thick] (0,.5) arc(90:180:.25cm) -- (-.25,-.75) arc (180:270:.25cm) -- (1.25,-1) arc(-90:0:.25cm) -- (1.5,0);
\draw[Box] (.25,-.75) rectangle (1.25,-.25); \node[scale=.5] at (.75,-.5) {$\sum TL$}; 
\draw[thickline] (2,0) -- (2,-1) arc(0:-90:.25cm) -- (-.25,-1.25) arc (270:180:.25cm) -- node[count,near end,scale=.7] {$k$} (-.5,1.5); \draw[verythickline] (.75,0) -- (.75,-.25); \draw[verythickline] (1.25,1) -- (1.25,1.5);
\node[left] at (-1,-.25) {$-\displaystyle \sum_{\substack{0 \leq k \leq t\\ k +s \in 2 \Z}}$};
\end{scope}
\end{tikzpicture}
\end{equation*}
\end{proposition}

\begin{proof}
Fix $x \in A$ and look at
\begin{equation*}
 \begin{tikzpicture}[scale=.75]
\draw[Box] (0,0) rectangle (2,1); \node at (1,.5){$x$};
\draw[Box] (.5,1.5) rectangle (4.5,2); \draw[verythickline] (1,1) -- (1,1.5);
\node[scale=.75] at (2.5,1.75) {$\sum TL$};
\begin{scope}[xshift=3cm]
\draw[Box] (0,0) rectangle (2,1); \node at (1,.5) {$Q$};
\draw[verythickline] (1,1) -- (1,1.5); 
\draw[thick] (0,.5) arc(270:180:.25cm) -- (-.25,1.5);
\draw[thickline] (1,0) arc(0:-90:.25cm) -- (-.25,-.25) arc (270:180:.25cm) -- (-.5,1.5);  
\end{scope}
 \end{tikzpicture}
\end{equation*}
Consider the string on the left side of $Q$, for a fixed term in the sum over $TL$ diagrams there are three possibilities: this string is connected to the top of $Q$, to the bottom of $Q$, or to $x$.  It follows that the inner product of $x$ with the element in the statement of the proposition is equal to the sum of the terms above in which the string on the left side of $Q$ is connected to $x$.  But it is clear that this is equal to $(\tau_k \boxtimes \tau_k) \widetilde \delta_Q(x) = \langle \delta_Q(x), 1 \otimes 1 \rangle$, which completes the proof.  
\end{proof}

\begin{corollary}
If $Q \in \Phi$ then $\delta_Q$ is closable as an unbounded operator $L^2(M) \to L^2(M \otimes M^{op})$.
\end{corollary}
\begin{proof}
This follows from $1 \otimes 1 \in \frk D(\delta_Q^*)$ by \cite{voi2}.
\end{proof}

Note that $\Phi$ is identified with $(p_{1,+} + p_{1,-})Vp_{0,+}$, and so may be completed to obtain a Hilbert space $\overline \Phi \simeq (p_{1,+} + p_{1,-})L^2(\frk M)p_{0,+}$.  Likewise $\Omega$ may completed to obtain a Hilbert space $\overline \Omega \simeq (p_{0,+} + p_{0,-})L^2(\frk M)p_{0,+}$.

\begin{proposition}
As right $M \otimes M^{op}$ modules, we have
\begin{equation*}
 \dim_{M \otimes M^{op}}(\overline \Phi) = 2\delta I,
\end{equation*}
and
\begin{equation*}
\dim_{M \otimes M^{op}}(\overline \Omega) = 2I.
\end{equation*}
\end{proposition}

\begin{proof}
Let $c \in V_{0,1}^-(1,0)$ be the following element:
\begin{equation*}
\begin{tikzpicture}[xscale=.6,yscale=.6]
\draw[Box] (0,0) rectangle (1,1);
\draw[thick] (1,.5) arc(270:180:.5cm); 
\node[left] at (-.4,.5) {$c =$};
\end{tikzpicture}
\end{equation*}
It follows from the same argument as in Lemma \ref{compress} that $L^2(M \boxtimes M) \simeq q\cdot \overline {\Phi_o}$ as right $M \otimes M^{op}$-modules, where $q = c^*(cc^*)^{-1}c$ is the initial projection of $c^*c$.  Note that there is an obvious inclusion of $M_1$ into the commutant of the right action of $M \otimes M^{op}$ on $\overline \Phi_o$.  It follows from the uniqueness of the trace on $M_1$ that the trace of $q$ in the commutant of $M \otimes M^{op}$ is equal to its trace in $M_1$, which is $\delta^{-1}$.  By Corollary \ref{index} we then have 
\begin{equation*}
\dim_{M \otimes M^{op}} \overline {\Phi_o} = \delta \cdot I.
\end{equation*}
Likewise we have $\dim_{M \otimes M^{op}} \overline{\Phi_e} = \delta \cdot I$, so that $\dim_{M \otimes M^{op}} \overline{\Phi} = 2\delta I$.

Clearly we have $\overline{\Omega_e} \simeq L^2(M \boxtimes M^{op})$, and hence $\dim_{M \otimes M^{op}}\overline{\Omega_e} = I$.  On the other hand, let $c' \in V_{0,1}^+(1,0)$ be as $c$ above but with the shading reversed, and set $q' = {c'}^*(c{c'}^*)^{-1}c'$.  Then by the same argument as Lemma \ref{compress}, we have $\overline{\Omega_o} \simeq q' \cdot \overline{\Phi_e}$, and hence $\dim_{M \otimes M^{op}} \overline{\Omega_o} = I$.
\end{proof}

\begin{remark}
For a family $X_1,\dotsc,X_n$ of self-adjoint operators which generate a II$_1$ factor $M$, Voiculescu has defined the \textit{free entropy dimension} $\delta_0(X_1,\dotsc,X_n)$.  If $M$ is isomorphic to an interpolated free group factor $L\mb F_t$, it is generally expected (but not proven) that $\delta_0(X_1,\dotsc,X_n) = t$.

Beginning with the work of Connes-Shlyakhtenko on $L^2$-Betti numbers \cite{cos}, there have been a number of recent results relating Voiculescu's free entropy dimension to the Murray von Neumann dimension of certain spaces of derivations.  In particular, it it reasonable to expect that in ``nice'' situations, we should have the following equality:
\begin{equation*}
\dim_{M \otimes M^{op}} \overline{\Der_c(A,A \otimes A^{op})}^{(L^2(M \otimes M^{op}))^n} = \delta_0(X_1,\dotsc,X_n),
\end{equation*}
where $A = \C \langle X_1,\dotsc,X_n \rangle$ and $\Der_c(A,A \otimes A^{op})$ denotes the module of closable derivations $\delta:A \to A \otimes A^{op}$, which is embedded into $(L^2(M \otimes M^{op}))^n$ via $\delta \mapsto (\delta(X_1),\dotsc,\delta(X_n))$.  Note in particular that if $X_1,\dotsc,X_n$ are free semicircular random variables, then both sides of the equation are indeed equal to $n$.

Combining the results above, we have constructed a $A \otimes A^{op}$-linear map from $D = \{\delta_Q:Q \in \Phi\}$ onto a dense subset of the orthogonal complement in $\overline \Phi$ of $\rho(\overline \Omega)$.  It follows that
\begin{equation*}
\dim_{M \otimes M^{op}} \overline{D}^{\overline \Phi} = \dim_{M \otimes M^{op}} \overline \Phi - \dim_{M \otimes M^{op}} \rho(\overline \Omega \ominus L^2(M \otimes M^{op})).
\end{equation*}
One can show that the extension of $\rho$ to $\overline \Omega$ is still injective, but the proof is rather tedious and so we omit it.  Assuming this, we are led to the equation
\begin{equation*}
 \dim_{M \otimes M^{op}} \overline{D}^{\overline \Phi} = 2\delta I - (2I - 1) = 1 + 2I(\delta-1) = r_0.
\end{equation*}
Now if we have $D = \Der_c(A,A \otimes A^{op})$, then in light of the remark above this ``explains'' the formula for $r_0$ from \cite{gjs2}.  Similar reasoning may of course be applied to the parameters $r_k$, $k \geq 1$.

So we are led to the question of whether every closable derivation $\delta:A \to A \otimes A^{op}$ is equal to $\delta_Q$ for some $Q \in \Phi$.  First note that it is not the case that every (algebraic) derivation $A \to A \otimes A^{op}$ is of this form.  For example, when $\mc P$ is the tensor planar algebra on $\C^n$ then $A \simeq \C \langle t_1,\dotsc,t_{n^2}\rangle$ is the free algebra generated by $P_1$.  There are then the free difference quotient derivations $\partial_{i}:A \to A \otimes A^{op}$ defined by
\begin{equation*}
\partial_{i}(t_j) = \delta_{ij} \cdot 1 \otimes 1.
\end{equation*}
Note that $\partial_i$ takes the $n$ graded component of $A$ into the $n-1$ graded component of $A \otimes A^{op}$.  But it is clear that the derivations $\delta_Q$ never decrease the grading, and so we cannot have $\partial_i = \delta_Q$.

On the other hand, we will now show that any ``sufficiently smooth'' derivation is equal to some $\delta_Q$.  Here ``sufficiently smooth'' will mean that the derivation can be extended to the semifinite algebra from \cite{gjs2}.  As we will discuss below, we believe that this condition should be equivalent to closability.
\end{remark}

We will use the construction from \cite{gjs2}, the reader is referred there for details.  To avoid conflict with the notations of the current paper, we will use $\mc A_+, \mc V_+$ to denote the algebras $A_+,V_+$ from \cite{gjs2}.  Note that we have an identification $A = e_0\mc V_+e_0$.

\begin{proposition}\label{extension}
Let $\widetilde \delta:\mc V_+ \to \mc V_+e_0 \otimes e_0 \mc V_+^{op}$ be a derivation such that $\widetilde \delta_{\mc A_+} = 0$.  Let $\delta:A \to A \otimes A^{op}$ denote the restriction of $\widetilde \delta$, then $\delta = \delta_Q$ for some $Q \in \Phi$.
\end{proposition}

\begin{proof}
Recall from \cite{gjs2} that $\mc V_+$ is generated by $\mc A_+$ and $\{c_n,c_n^*:n \geq 0\}$, where
\begin{equation*}
\begin{tikzpicture}
\draw[Box] (0,1) rectangle (1,2);
\draw[thick] (1,1.6) arc(-90:-180:.5cm and .4cm);
\draw[thickline] (0,1.25) -- node[count,rectangle,rounded corners] {$2n$} (1,1.25);
\node[left] at (-.25,1.5) {$c_n = $};
\end{tikzpicture}
\end{equation*}
Let 
\begin{align*}
\widetilde \delta(c_n) &= \sum_{\substack{\text{$s$ even}\\\text{$t$ odd}}} a_{2n,s} \otimes b_{2n+1,t}^{op},\\
\widetilde \delta(c_n^*) &= \sum_{\substack{\text{$s$ odd}\\\text{$t$ even}}} a_{2n+1,s} \otimes b_{2n,t}^{op},
\end{align*}
where $a_{m,s} \otimes b_{m,t}^{op} \in V_{m,0}^{(-1)^s}(s,0) \otimes V_{0,m}^+(t,0)$ (a slight abuse of notation as this need not be a pure tensor).

Let $p \geq 0, q > 0$ be even with $p+q = 2n$. We claim that for any $y \in P_n$ and any even $s$ and odd $t$ we have the following relation:
\begin{equation*}
\begin{tikzpicture}
\begin{scope}[yscale=.75]
\draw[Box] (0,1.5) rectangle (1,2.5); \node[scale=.75] at (.5,2) {$a_{q,s}$};
\draw[thickline] (.5,2.5) -- (.5,2.75); 
\draw[Box] (-2,1.5) rectangle (-1,2.5); \node at (-1.5,2) {$y$}; \node[marked] at (-2.1,1.4) {};
\draw[thickline] (-1,2) -- node[rcount] {$q$} (0,2); \draw[thickline](-2,2) arc(270:180:.25cm) -- (-2.25,2.75);
\draw[Box] (0,0) rectangle (1,1); \node[scale=.75] at (.5,.5) {$b_{q+1,t}$};
\draw[thickline] (.5,0) -- (.5,-.25); \draw[thickline] (0,.75) arc(90:180:.5cm) -- (-.5,-.25);
\draw[thick] (0,.25) arc(270:90:.25cm and .125cm);
\end{scope}
\begin{scope}[xshift=4.5cm]
\begin{scope}[yscale=.75]
\draw[Box] (0,1.5) rectangle (1,2.5); \node[scale=.75] at (.5,2) {$a_{p,s}$};
\draw[thickline] (.5,2.5) -- (.5,2.75); 
\draw[Box] (-2,0) rectangle (-1,1); \node[flip] at (-1.5,.5) {$y$}; \node[marked] at (-.9,1.1) {};
\draw[thickline] (-1,.5) -- node[rcount,scale=.7] {$p+1$} (0,.5); 
\draw[thickline](-2,.5) arc(90:180:.25cm) -- (-2.25,-.25);
\draw[Box] (0,0) rectangle (1,1); \node[scale=.75] at (.5,.5) {$b_{p+1,t}$};
\draw[thickline] (.5,0) -- (.5,-.25); \draw[thickline] (0,2) arc(270:180:.25cm) -- (-.25,2.75);
\end{scope}
\node[left] at (-2.5,1) {$=$};
\end{scope}
\end{tikzpicture}
\end{equation*}

To prove this we first note that it suffices to show this for $q$ sufficiently large, which follows easily from $c_n = \delta\cdot e_{2n}c_{n+1}$.  We therefore assume that $q > 2s$, which will simplify the formulas which follow.

 It follows from routine calculations that
\begin{align*}
\begin{tikzpicture}[yscale=.75]
\draw[Box] (0,1.5) rectangle (1.5,2.5); \node[scale=.75] at (.75,2) {$a_{2(n-k),s}$};
\draw[thickline] (.75,2.5) -- (.75,2.75); \draw[thickline] (0,2) arc(270:180:.25cm) -- (-.25,2.75);
\draw[Box] (0,0) rectangle (1.5,1); \node[scale=.75] at (.75,.5) {$b_{2(n-k)+1,t}$};
\draw[thickline] (.75,0) -- (.75,-.25); \draw[thickline] (0,.6) -- (-1.25,.6);
\draw[thickline] (-1.25,.25) -- node[rcount,scale=.6] {$2k-1$} (-.25,.25) arc(90:0:.2cm and .5cm);
\node[left] at (-1.5,.85) {$\displaystyle\sum_{\substack{s,t\\\text{$s$ even}\\\text{$t$ odd}}}\displaystyle\sum_{k=1}^{n}$};
\begin{scope}[xshift=5cm]
\draw[Box] (0,1.5) rectangle (1.5,2.5); \node[scale=.75] at (.75,2) {$a_{2(n-k)+1,s}$};
\draw[thickline] (.75,2.5) -- (.75,2.75); \draw[thickline] (0,1.75) arc(270:180:.5cm and .25cm) -- (-.5,2.75);
\draw[thick] (0,2) arc(270:90:.25cm and .125cm);
\draw[Box] (0,0) rectangle (1.5,1); \node[scale=.75] at (.75,.5) {$b_{2(n-k),t}$};
\draw[thickline] (.75,0) -- (.75,-.25); \draw[thickline] (0,.6) -- (-1,.6);
\draw[thickline] (-1,.25) -- node[rcount,scale=.6] {$2k$} (-.25,.25) arc(90:0:.2cm and .5cm);
\node[left] at (-1.25,.85) {$+\displaystyle\sum_{\substack{s,t\\\text{$s$ odd}\\\text{$t$ even}}}\displaystyle\sum_{k=0}^{n-1}$};
\end{scope}
\begin{scope}[xshift=-5cm,yshift=.75cm]
\draw[thickline] (0,1) arc(180:270:.5cm and 1cm); \node[scale=.7,rcount] at (.1,.5) {$2n$};
\node[left] at (-.1,.4) {$\widetilde \delta\biggl($}; \node[right] at (.6,.4) {$\biggr) =$};
\end{scope}
\end{tikzpicture}\\
\begin{tikzpicture}[yscale=.75]
\draw[Box] (0,1.5) rectangle (1.5,2.5); \node[scale=.75] at (.75,2) {$a_{2(n-k)+1,s}$};
\draw[thickline] (.75,2.5) -- (.75,2.75); \draw[thickline] (0,.5) arc(90:180:.25cm) -- (-.25,-.25);
\draw[thickline] (0,1.85) -- (-1.25,1.85);
\draw[Box] (0,0) rectangle (1.5,1); \node[scale=.75] at (.75,.5) {$b_{2(n-k),t}$};
\draw[thickline] (.75,0) -- (.75,-.25); 
\draw[thickline] (-1.25,2.25) -- node[rcount,scale=.6] {$2k-1$} (-.25,2.25) arc(-90:0:.2cm and .5cm);
\node[left] at (-1.5,.85) {$\displaystyle\sum_{\substack{s,t\\\text{$s$ odd}\\\text{$t$ even}}}\displaystyle\sum_{k=1}^{n}$};
\begin{scope}[xshift=5cm]
\draw[Box] (0,1.5) rectangle (1.5,2.5); \node[scale=.75] at (.75,2) {$a_{2(n-k),s}$};
\draw[thickline] (.75,2.5) -- (.75,2.75); \draw[thickline] (0,.75) arc(90:180:.5cm) -- (-.5,-.25);
\draw[thickline] (0,1.85) -- (-1,1.85); \draw[thickline] (-1,2.25) -- node[rcount,scale=.6] {$2k$} (-.25,2.25) arc(-90:0:.2cm and .5cm);
\draw[thick] (0,.25) arc(270:90:.25cm and .125cm);
\draw[Box] (0,0) rectangle (1.5,1); \node[scale=.75] at (.75,.5) {$b_{2(n-k)+1,t}$};
\draw[thickline] (.75,0) -- (.75,-.25);
\node[left] at (-1.25,.85) {$+\displaystyle\sum_{\substack{s,t\\\text{$s$ even}\\\text{$t$ odd}}}\displaystyle\sum_{k=0}^{n-1}$};
\end{scope}
\begin{scope}[xshift=-5cm,yshift=.75cm]
\begin{scope}[xscale=-1,xshift=-.5cm]
\draw[thickline] (0,1) arc(180:270:.5cm and 1cm); \node[scale=.7,rcount] at (.1,.5) {$2n$};
\end{scope}
\node[left] at (-.1,.4) {$\widetilde \delta\biggl($}; \node[right] at (.6,.4) {$\biggr) =$};
\end{scope}
\end{tikzpicture}
\end{align*}

Now let $R_{s,t}$ denote the component of $\widetilde \delta(y)$ which lies in $V_{0,0}(p+2s,q-1+t)$.  Since we have
\begin{equation*}
\begin{tikzpicture}[scale=.75]
\draw[Box] (0,0) rectangle (1,1); \node at (.5,.5) {$y$};
\draw[thickline] (0,.5) arc(270:180:.5cm and .75cm); \node[scale=.75,rcount] at (-.4,.8) {$p$};
\draw[thickline] (1,.5) arc(-90:0:.5cm and .75cm); \node[scale=.75,rcount] at (1.4,.8) {$q$};
\node[left] at (-.55,.5) {$\widetilde\delta(y) = \widetilde \delta\Biggl($};
\node[right] at (1.55,.5) {$\Biggr),$};
\end{tikzpicture}
\end{equation*}
it follows from the formulas above that
\begin{align*}
\begin{tikzpicture}[yscale=.75]
\draw[Box] (0,1.5) rectangle (1,2.5); \node[scale=.75] at (.5,2) {$a$};
\draw[Box] (-1.75,1.5) rectangle (-1.25,2.5); \node[scale=.75] at (-1.5,2) {$y$};
\draw[thickline] (-1.75,2) arc(270:180:.25cm and .5cm) -- node[rcount,scale=.75] {$p$} (-2,2.75);
\draw[thickline] (.5,2.5) -- (.5,2.75); \draw[thickline] (0,.5) arc(90:180:.5cm) -- node[rcount,scale=.6,near end] {$q-2k$} (-.5,-.25);
\draw[thickline] (0,1.85) -- (-1.25,1.85);
\draw[Box] (0,0) rectangle (1,1); \node[scale=.75] at (.5,.5) {$b$};
\draw[thickline] (.5,0) -- (.5,-.25); 
\draw[thickline] (-1.25,2.25) -- node[rcount,scale=.6] {$2k-1$} (-.25,2.25) arc(-90:0:.2cm and .5cm);
\node[left] at (-2.25,1.25) {$R_{s,t} = \displaystyle\sum_{k=1}^{s}$};
\begin{scope}[xshift=4.5cm]
\draw[Box] (0,1.5) rectangle (1,2.5); \node[scale=.75] at (.5,2) {$a$};
\draw[Box] (-1.5,1.5) rectangle (-1,2.5); \node[scale=.75] at (-1.25,2) {$y$};
\draw[thickline] (-1.5,2) arc(270:180:.25cm and .5cm) -- node[rcount,scale=.75] {$p$} (-1.75,2.75);
\draw[thickline] (.5,2.5) -- (.5,2.75); \draw[thickline] (0,.75) arc(90:180:.75cm and .5cm) -- node[rcount,scale=.6,near end] {$q-2k+1$} (-.75,-.25);
\draw[thickline] (0,1.85) -- (-1,1.85); \draw[thickline] (-1,2.25) -- node[rcount,scale=.6] {$2k$} (-.25,2.25) arc(-90:0:.2cm and .5cm);
\draw[thick] (0,.25) arc(270:90:.25cm and .125cm);
\draw[Box] (0,0) rectangle (1,1); \node[scale=.75] at (.5,.5) {$b$};
\draw[thickline] (.5,0) -- (.5,-.25);
\node[left] at (-1.75,1.25) {$+\displaystyle\sum_{k=0}^{s}$};
\end{scope}
\end{tikzpicture}\\
\begin{tikzpicture}[yscale=.75]
\draw[Box] (0,1.5) rectangle (1,2.5); \node[scale=.75] at (.5,2) {$a$};
\draw[thickline] (.5,2.5) -- (.5,2.75); \draw[thickline] (0,2) arc(270:180:.5cm) -- node[rcount,scale=.6,near end] {$2k-2$} (-.5,2.75);
\draw[Box] (0,0) rectangle (1,1); \node[scale=.75] at (.5,.5) {$b$};
\draw[Box] (-1.75,0) rectangle(-1.25,1); \node[scale=.75] at (-1.5,.5) {$y$};
\draw[thickline] (-1.75,.5) arc(90:180:.25cm and .5cm) -- node[rcount,scale=.75] {$q$}(-2,-.25);
\draw[thickline] (.5,0) -- (.5,-.25); \draw[thickline] (0,.6) -- node[rcount,scale=.6] {$2k-1$}  (-1.25,.6);
\draw[thickline] (-1.25,.25) --(-.25,.25) arc(90:0:.2cm and .5cm);
\node[left] at (-2.25,1.25) {$+ \displaystyle\sum_{k=1}^{p/2}$};
\begin{scope}[xshift=4.5cm]
\draw[Box] (0,1.5) rectangle (1,2.5); \node[scale=.75] at (.5,2) {$a$};
\draw[thickline] (.5,2.5) -- (.5,2.75); \draw[thickline] (0,1.75) arc(270:180:.5cm and .25cm) -- node[rcount,scale=.6,near end] {$2k-1$} (-.5,2.75);
\draw[thick] (0,2) arc(270:90:.25cm and .125cm);
\draw[Box] (0,0) rectangle (1,1); \node[scale=.75] at (.5,.5) {$b$};
\draw[Box] (-1.5,0) rectangle(-1,1); \node[scale=.75] at (-1.25,.5) {$y$};
\draw[thickline] (-1.5,.5) arc(90:180:.25cm and .5cm) -- node[rcount,scale=.75] {$q$} (-1.75,-.25);
\draw[thickline] (.5,0) -- (.5,-.25); \draw[thickline] (0,.6) -- node[rcount,scale=.6] {$2k$} (-1,.6);
\draw[thickline] (-1,.25) -- (-.25,.25) arc(90:0:.2cm and .5cm);
\node[left] at (-2,1.25) {$+\displaystyle\sum_{k=1}^{p/2}$};
\end{scope}
\end{tikzpicture}
\end{align*}
Note that we have suppressed the subscripts on $a$ and $b$, which are determined by the condition that each term has $p+2s$ upper strings and $q-1+t$ lower strings.

On the other hand, we also have
\begin{equation*}
 \begin{tikzpicture}[scale=.75]
\draw[Box] (0,0) rectangle (1,1); \node at (.5,.5) {$y$};
\draw[thickline] (0,.5) arc(270:180:.5cm and .75cm); \node[scale=.6,rcount] at (-.5,.9) {$p+1$};
\draw[thickline] (1,.5) arc(-90:0:.5cm and .75cm); \node[scale=.6,rcount] at (1.5,.9) {$q-1$};
\node[left] at (-.65,.5) {$\widetilde\delta(y) = \widetilde \delta\Biggl($};
\node[right] at (1.65,.5) {$\Biggr).$};
\end{tikzpicture}
\end{equation*}
It then follows by similar computations that
\begin{align*}
\begin{tikzpicture}[yscale=.75]
\draw[Box] (0,1.5) rectangle (1,2.5); \node[scale=.75] at (.5,2) {$a$};
\draw[Box] (-1.75,1.5) rectangle (-1.25,2.5); \node[scale=.75] at (-1.5,2) {$y$};
\draw[thickline] (-1.75,2) arc(270:180:.25cm and .5cm) -- node[rcount,scale=.6] {$p+1$} (-2,2.75);
\draw[thickline] (.5,2.5) -- (.5,2.75); \draw[thickline] (0,.5) arc(90:180:.5cm) -- node[rcount,scale=.6,near end] {$q-2k-1$} (-.5,-.25);
\draw[thickline] (0,1.85) -- (-1.25,1.85);
\draw[Box] (0,0) rectangle (1,1); \node[scale=.75] at (.5,.5) {$b$};
\draw[thickline] (.5,0) -- (.5,-.25); 
\draw[thickline] (-1.25,2.25) -- node[rcount,scale=.6] {$2k-1$} (-.25,2.25) arc(-90:0:.2cm and .5cm);
\node[left] at (-2.25,1.25) {$R_{s,t} = \displaystyle\sum_{k=1}^{s}$};
\begin{scope}[xshift=4.5cm]
\draw[Box] (0,1.5) rectangle (1,2.5); \node[scale=.75] at (.5,2) {$a$};
\draw[Box] (-1.5,1.5) rectangle (-1,2.5); \node[scale=.75] at (-1.25,2) {$y$};
\draw[thickline] (-1.5,2) arc(270:180:.25cm and .5cm) -- node[rcount,scale=.6] {$p+1$} (-1.75,2.75);
\draw[thickline] (.5,2.5) -- (.5,2.75); \draw[thickline] (0,.75) arc(90:180:.75cm and .5cm) -- node[rcount,scale=.6,near end] {$q-2k-2$} (-.75,-.25);
\draw[thickline] (0,1.85) -- (-1,1.85); \draw[thickline] (-1,2.25) -- node[rcount,scale=.6] {$2k$} (-.25,2.25) arc(-90:0:.2cm and .5cm);
\draw[thick] (0,.25) arc(270:90:.25cm and .125cm);
\draw[Box] (0,0) rectangle (1,1); \node[scale=.75] at (.5,.5) {$b$};
\draw[thickline] (.5,0) -- (.5,-.25);
\node[left] at (-1.75,1.25) {$+\displaystyle\sum_{k=0}^{s-1}$};
\end{scope}
\end{tikzpicture}\\
\begin{tikzpicture}[yscale=.75]
\draw[Box] (0,1.5) rectangle (1,2.5); \node[scale=.75] at (.5,2) {$a$};
\draw[thickline] (.5,2.5) -- (.5,2.75); \draw[thickline] (0,2) arc(270:180:.5cm) -- node[rcount,scale=.6,near end] {$2k$} (-.5,2.75);
\draw[Box] (0,0) rectangle (1,1); \node[scale=.75] at (.5,.5) {$b$};
\draw[Box] (-1.75,0) rectangle(-1.25,1); \node[scale=.75] at (-1.5,.5) {$y$};
\draw[thickline] (-1.75,.5) arc(90:180:.25cm and .5cm) -- node[rcount,scale=.6] {$q-1$}(-2,-.25);
\draw[thickline] (.5,0) -- (.5,-.25); \draw[thickline] (0,.6) -- node[rcount,scale=.6] {$2k+1$}  (-1.25,.6);
\draw[thickline] (-1.25,.25) --(-.25,.25) arc(90:0:.2cm and .5cm);
\node[left] at (-2.25,1.25) {$+ \displaystyle\sum_{k=0}^{p/2}$};
\begin{scope}[xshift=4.5cm]
\draw[Box] (0,1.5) rectangle (1,2.5); \node[scale=.75] at (.5,2) {$a$};
\draw[thickline] (.5,2.5) -- (.5,2.75); \draw[thickline] (0,1.75) arc(270:180:.5cm and .25cm) -- node[rcount,scale=.6,near end] {$2k-1$} (-.5,2.75);
\draw[thick] (0,2) arc(270:90:.25cm and .125cm);
\draw[Box] (0,0) rectangle (1,1); \node[scale=.75] at (.5,.5) {$b$};
\draw[Box] (-1.5,0) rectangle(-1,1); \node[scale=.75] at (-1.25,.5) {$y$};
\draw[thickline] (-1.5,.5) arc(90:180:.25cm and .5cm) -- node[rcount,scale=.6] {$q-1$} (-1.75,-.25);
\draw[thickline] (.5,0) -- (.5,-.25); \draw[thickline] (0,.6) -- node[rcount,scale=.6] {$2k$} (-1,.6);
\draw[thickline] (-1,.25) -- (-.25,.25) arc(90:0:.2cm and .5cm);
\node[left] at (-2,1.25) {$+\displaystyle\sum_{k=1}^{p/2}$};
\end{scope}
\end{tikzpicture}
\end{align*}

Setting these two expressions for $R_{s,t}$ equal to each other and cancelling out common terms, we obtain the desired formula.

By the same argument, if $p,q \geq 1$ are odd and $p+q = 2n$, then for any $y \in P_{n}$ and any odd $s$ and even $t$ we have:
\begin{equation*}
\begin{tikzpicture}
\begin{scope}[yscale=.75]
\draw[Box] (0,1.5) rectangle (1,2.5); \node[scale=.75] at (.5,2) {$a_{q,s}$};
\draw[thickline] (.5,2.5) -- (.5,2.75); 
\draw[Box] (-2,1.5) rectangle (-1,2.5); \node at (-1.5,2) {$y$}; \node[marked] at (-2.1,1.4) {};
\draw[thickline] (-1,2) -- node[rcount] {$q$} (0,2); \draw[thickline](-2,2) arc(270:180:.25cm) -- (-2.25,2.75);
\draw[Box] (0,0) rectangle (1,1); \node[scale=.75] at (.5,.5) {$b_{q-1,t}$};
\draw[thickline] (.5,0) -- (.5,-.25); \draw[thickline] (0,.5) arc(90:180:.25cm) -- (-.25,-.25);
\end{scope}
\begin{scope}[xshift=4.5cm]
\begin{scope}[yscale=.75]
\draw[Box] (0,1.5) rectangle (1,2.5); \node[scale=.75] at (.5,2) {$a_{p+2,s}$};
\draw[thickline] (.5,2.5) -- (.5,2.75); 
\draw[Box] (-2,0) rectangle (-1,1); \node[flip] at (-1.5,.5) {$y$}; \node[marked] at (-.9,1.1) {};
\draw[thickline] (-1,.5) -- node[rcount,scale=.7] {$p$} (0,.5); 
\draw[thickline](-2,.5) arc(90:180:.25cm) -- (-2.25,-.25);
\draw[Box] (0,0) rectangle (1,1); \node[scale=.75] at (.5,.5) {$b_{p,t}$};
\draw[thickline] (.5,0) -- (.5,-.25); \draw[thickline] (0,1.75) arc(270:180:.5cm) -- (-.5,2.75);
\draw[thick] (0,2) arc(270:90:.25cm and .125cm);
\end{scope}
\node[left] at (-2.5,1) {$=$};
\end{scope}
\end{tikzpicture}
\end{equation*}

Now for $s \geq 1$ odd and $t \geq 0$ even, define $Q_{s,t} \in V_{1,0}^-(s,t)$ by
\begin{equation*}
\begin{tikzpicture}
\draw[Box] (0,0) rectangle (1,1); \node at (.5,.5) {$Q_{s,t}$};
\draw[thick] (0,.5) arc(-90:-180:.25cm and .5cm) -- (-.25,1.25);
\draw[thickline] (.5,1) -- (.5,1.25);
\draw[thickline] (.5,0) -- (.5,-.25);
\begin{scope}[xshift=2.25cm]
\draw[Box] (0,0) rectangle (1,1); \node[scale=.75] at (.5,.5) {$a_{s+2,s}$};
\draw[thickline] (0,.25) arc(270:90:.4cm and .5cm) -- ++(.25,0) arc(90:0:.25cm and .25cm);
\draw[thick] (0,.5) arc(270:90:.25cm and .125cm);
\node[left] at (-.5,.5) {$=$};
\end{scope}
\begin{scope}[xshift=4.25cm]
\draw[Box] (0,0) rectangle (1,1); \node[scale=.75] at (.5,.5) {$b_{s+1,t}$};
\draw[thickline] (.5,0) -- (.5,-.25);
\draw[thickline] (0,.5) arc(270:180:.25cm) -- (-.25,1.25);
\node[left] at (-.5,.5) {$\cdot$};
\end{scope}
\end{tikzpicture}
\end{equation*}
We claim that
\begin{equation*}
\begin{tikzpicture}
\begin{scope}[yscale=.75]
\draw[Box] (0,1.5) rectangle (1,2.5); \node[scale=.75] at (.5,2) {$a_{2n+1,s}$};
\draw[thickline] (.5,2.5) -- (.5,2.75); \draw[thickline] (0,2) arc(270:180:.25cm) -- (-.25,2.75);
\draw[Box] (0,0) rectangle (1,1); \node[scale=.75] at (.5,.5) {$b_{2n,t}$};
\draw[thickline] (.5,0) -- (.5,-.25); \draw[thickline] (0,.5) arc(90:180:.25cm) -- (-.25,-.25);
\end{scope}
\begin{scope}[xshift=4cm,yshift=.4cm]
\node[left] at (0,.5) {$= E_{M \otimes M^{op}}\Biggl[$};
\draw[thickline] (.25,-.25) -- node[count,rectangle,rounded corners] {$2n$} (.25,1.25);
\begin{scope}[xshift=1cm]
\draw[Box] (0,0) rectangle (1,1); \node at (.5,.5) {$Q_{s,t}$};
\draw[thick] (0,.5) arc(-90:-180:.25cm and .5cm) -- (-.25,1.25);
\draw[thickline] (.5,1) -- (.5,1.25);
\draw[thickline] (.5,0) -- (.5,-.25);
\node[right] at (1.25,.5) {$\Biggr]$};
\end{scope}
\end{scope}
\end{tikzpicture}
\end{equation*}
Indeed, using the relation above we see that for $x \in V_{0,2n+1}(s,0)$ and $y \in V_{t,0}(2n,0)$ we have:
\begin{equation*}
\begin{tikzpicture}[yscale=.75]
\draw[Box] (0,1.5) rectangle (1,2.5); \node[scale=.75] at (.5,2) {$a_{2n+1,s}$};
\draw[thickline] (.5,2.5) arc(180:90:.25cm) -- (1.75,2.75) arc(90:0:.25cm); \draw[thickline] (0,2) arc(270:180:.25cm) -- (-.25,2.85) arc(180:90:.25cm) -- (2.5,3.1) arc(90:0:.25cm) -- (2.75,2.25) arc(0:-90:.25cm);
\draw[Box] (0,0) rectangle (1,1); \node[scale=.75] at (.5,.5) {$b_{2n,t}$};
\draw[thickline] (.5,0) arc(180:270:.25cm) -- (1.75,-.25) arc(-90:0:.25cm); \draw[thickline] (0,.5) arc(90:180:.25cm) -- (-.25,-.35) arc(180:270:.25cm) -- (2.5,-.6) arc(-90:0:.25cm) -- (2.75,.25) arc(0:90:.25cm);
\draw[Box] (1.5,1.5) rectangle (2.5,2.5); \node[scale=.75] at (2,2) {$x$};
\draw[Box] (1.5,0) rectangle (2.5,1); \node[scale=.75] at (2,.5) {$y$};
\begin{scope}[xshift=6cm]
\draw[Box] (0,1.5) rectangle (1,2.5); \node[scale=.75] at (.5,2) {$a_{s+2,s}$};
\draw[thick] (0,2) arc(270:90:.25cm and .125cm);
\draw[thickline] (0,1.75) arc(270:180:.5cm) -- (-.5,2.5) arc(180:0:.5cm); 
\draw[Box] (-1.375,-.15) rectangle (-.625,1.15); \node[scale=.75,rotate=270] at (-1,.5) {$x$};
\draw[thickline] (-.625,.65)  -- (0,.65); \draw[thick](-.95,-.15) arc(180:360:.3cm and .15cm) arc(180:90:.35cm and .45cm);
\draw[thickline] (-1.15,-.15) arc(180:270:.6cm and .45cm) -- (2.5,-.6) arc(-90:90:.25cm and .55cm);
\draw[Box] (0,0) rectangle (1,1); \node[scale=.75] at (.5,.5) {$b_{s+1,t}$};
\draw[thickline] (.5,0) arc(180:270:.25cm) -- (1.75,-.25) arc(-90:0:.25cm); 
\draw[Box] (1.5,0) rectangle (2.5,1); \node[scale=.75] at (2,.5) {$y$};
\node[left] at (-2,1.25) {$=$};
\end{scope}
\begin{scope}[xshift=10.75cm]
\draw[Box] (0,.75) rectangle (1,1.75); \node at (.5,1.25) {$Q_{s,t}$};
\draw[Box] (1.5,1.5) rectangle (2.5,2.5); \node[scale=.75] at (2,2) {$x$};
\draw[thickline](2.5,1.8) arc(-90:90:.4cm and .7cm) -- (0,3.2) arc(90:180:.5cm) -- (-.5,0) arc(180:270:.55cm) 
-- (2.5,-.6) arc(-90:90:.25cm and .55cm); 
\draw[thickline] (.5,1.75) arc(180:90:.75cm and 1cm) arc(90:0:.5cm and .25cm);
\draw[thickline] (.5,.75) arc(180:270:1cm) arc(-90:0:.5cm and .25cm);
\draw[thick] (0,1.25) arc(270:180:.25cm) -- (-.25,2.7) arc(180:90:.25cm) -- (2.5,2.95) arc(90:0:.15cm) -- (2.65,2.3) arc(0:-90:.15cm);
\draw[Box] (1.5,0) rectangle (2.5,1); \node[scale=.75] at (2,.5) {$y$};
\node[left] at (-1,1.25) {$=$};
\end{scope}
\end{tikzpicture}
\end{equation*}
The result then follows from Lemma \ref{oproj}.

Likewise, for $s \geq 0$ even and $t \geq 1$ odd, define
\begin{equation*}
\begin{tikzpicture}
\draw[Box] (0,0) rectangle (1,1); \node at (.5,.5) {$Q_{s,t}$};
\draw[thick] (0,.5) arc(279:180:.25cm and .5cm) -- (-.25,1.25);
\draw[thickline] (.5,1) -- (.5,1.25);
\draw[thickline] (.5,0) -- (.5,-.25);
\begin{scope}[xshift=2.25cm]
\draw[Box] (0,0) rectangle (1,1); \node[scale=.75] at (.5,.5) {$a_{s,s}$};
\draw[thickline] (0,.5) arc(270:180:.25cm) -- ++(0,.25) arc(180:0:.375cm and .25cm);
\node[left] at (-.5,.5) {$=$};
\end{scope}
\begin{scope}[xshift=4.25cm]
\draw[Box] (0,0) rectangle (1,1); \node[scale=.75] at (.5,.5) {$b_{s+1,t}$};
\draw[thickline] (.5,0) -- (.5,-.25);
\draw[thickline] (0,.5) arc(270:180:.25cm) -- (-.25,1.25);
\node[left] at (-.5,.5) {$\cdot$};
\end{scope}
\end{tikzpicture}
\end{equation*}
A similar argument shows that
\begin{equation*}
\begin{tikzpicture}
\begin{scope}[yscale=.75]
\draw[Box] (0,1.5) rectangle (1,2.5); \node[scale=.75] at (.5,2) {$a_{2n,s}$};
\draw[thickline] (.5,2.5) -- (.5,2.75); \draw[thickline] (0,2) arc(270:180:.25cm) -- (-.25,2.75);
\draw[Box] (0,0) rectangle (1,1); \node[scale=.75] at (.5,.5) {$b_{2n+1,t}$};
\draw[thickline] (.5,0) -- (.5,-.25); \draw[thickline] (0,.75) arc(90:180:.5cm) -- (-.5,-.25);
\draw[thick] (0,.5) arc(90:270:.25cm and .125cm);
\end{scope}
\begin{scope}[xshift=4cm,yshift=.4cm]
\node[left] at (-.1,.5) {$= E_{M \otimes M^{op}}\Biggl[$};
\draw[thickline] (.25,-.25) -- node[count,rectangle,rounded corners,scale=.7] {$2n-1$} (.25,1.25);
\begin{scope}[xshift=1cm]
\draw[Box] (0,0) rectangle (1,1); \node at (.5,.5) {$Q_{s,t}$};
\draw[thick] (0,.5) arc(-90:-180:.25cm and .5cm) -- (-.25,1.25);
\draw[thickline] (.5,1) -- (.5,1.25);
\draw[thickline] (.5,0) -- (.5,-.25);
\node[right] at (1.25,.5) {$\Biggr]$};
\end{scope}
\end{scope}
\end{tikzpicture} 
\end{equation*}

Now set $Q = \sum_{s,t} Q_{s,t}$, we claim that $\delta = \delta_Q$, where $\delta$ is the restriction of $\widetilde \delta$ to $A = e_0\mc V_+ e_0$.  Indeed, for $x \in P_{n}$ we have
\begin{align*}
\begin{tikzpicture}
\begin{scope}[scale=.5,xshift=4.75cm,yshift=1.25cm]
\draw[Box] (0,0) rectangle (1,2); \node[rotate=270] at (.5,1) {$x$};
\draw[thickline] (1,1) arc(-90:0:.5cm and 1cm);
\node[left] at (-.25,1) {$\delta(x) = \widetilde\delta\Biggl($}; \node[right] at (1.75,1) {$\Biggr)$};
\end{scope}
\begin{scope}[xshift=8cm]
\draw[Box] (0,1.5) rectangle (1.5,2.5); \node[scale=.75] at (.75,2) {$a_{2(n-k)+1,s}$};
\draw[thickline] (.75,2.5) -- (.75,2.75); \draw[thickline] (0,.5) arc(90:180:.25cm) -- (-.25,-.25);
\draw[thickline] (0,1.85) -- (-1.25,1.85);
\draw[Box] (0,0) rectangle (1.5,1); \node[scale=.75] at (.75,.5) {$b_{2(n-k),t}$};
\draw[thickline] (.75,0) -- (.75,-.25); 
\draw[thickline] (-1.25,2.25) -- node[rcount,scale=.6] {$2k-1$} (-.25,2.25) arc(-90:0:.2cm and .5cm);
\draw[Box] (-1.75,1.5) rectangle (-1.25,2.5); \node[scale=.75,rotate=270] at (-1.5,2) {$x$};
\node[left] at (-2,.85) {$= \displaystyle\sum_{\substack{s,t\\\text{$s$ odd}\\\text{$t$ even}}}\displaystyle\sum_{k=1}^{n}$};
\begin{scope}[xshift=5.5cm]
\draw[Box] (0,1.5) rectangle (1.5,2.5); \node[scale=.75] at (.75,2) {$a_{2(n-k),s}$};
\draw[thickline] (.75,2.5) -- (.75,2.75); \draw[thickline] (0,.75) arc(90:180:.5cm) -- (-.5,-.25);
\draw[thickline] (0,1.85) -- (-1,1.85); \draw[thickline] (-1,2.25) -- node[rcount,scale=.6] {$2k$} (-.25,2.25) arc(-90:0:.2cm and .5cm);
\draw[thick] (0,.25) arc(270:90:.25cm and .125cm);
\draw[Box] (-1.5,1.5) rectangle (-1,2.5); \node[scale=.75,rotate=270] at (-1.25,2) {$x$};
\draw[Box] (0,0) rectangle (1.5,1); \node[scale=.75] at (.75,.5) {$b_{2(n-k)+1,t}$};
\draw[thickline] (.75,0) -- (.75,-.25);
\node[left] at (-1.75,.85) {$+\displaystyle\sum_{\substack{s,t\\\text{$s$ even}\\\text{$t$ odd}}}\displaystyle\sum_{k=0}^{n-1}$};
\end{scope}
\end{scope}
\end{tikzpicture}\\
\begin{tikzpicture}
\draw[thickline] (-1.25,2.25) -- node[rcount,scale=.6] {$2k-1$} (-.25,2.25) arc(-90:0:.2cm and .5cm);
\draw[Box] (-1.75,1.5) rectangle (-1.25,2.5); \node[scale=.75,rotate=270] at (-1.5,2) {$x$};
\draw[thickline] (-1.25,1.75) arc(90:0:.25cm) -- (-1,1);
\draw[Box] (0,1.25) rectangle (1,2.25); \node at (.5,1.75) {$Q_{s,t}$};
\draw[thick] (-1.25,2) .. controls ++(.5,0) and ++(-.5,0) .. (0,1.75);
\draw[thickline] (.5,2.25) -- (.5,2.75);
\draw[thickline] (.5,1) -- (.5,1.25);
\node[left] at (-1.85,1.5) {$=\displaystyle\sum_{\substack{s,t\\\text{$s$ odd}\\\text{$t$ even}}}\displaystyle\sum_{k=1}^{n}E_{M \otimes M^{op}}\Biggr[$}; \node[right] at (1.1,1.75) {$\Biggr]$};
\begin{scope}[xshift=7cm]
\draw[thickline] (-1.25,2.25) -- node[rcount,scale=.6] {$2k$} (-.25,2.25) arc(-90:0:.2cm and .5cm);
\draw[Box] (-1.75,1.5) rectangle (-1.25,2.5); \node[scale=.75,rotate=270] at (-1.5,2) {$x$};
\draw[thickline] (-1.25,1.75) arc(90:0:.25cm) -- (-1,1);
\draw[Box] (0,1.25) rectangle (1,2.25); \node at (.5,1.75) {$Q_{s,t}$};
\draw[thick] (-1.25,2) .. controls ++(.5,0) and ++(-.5,0) .. (0,1.75);
\draw[thickline] (.5,2.25) -- (.5,2.75);
\draw[thickline] (.5,1) -- (.5,1.25);
\node[left] at (-1.85,1.5) {$+\displaystyle\sum_{\substack{s,t\\\text{$s$ even}\\\text{$t$ odd}}}\displaystyle\sum_{k=0}^{n-1}E_{M \otimes M^{op}}\Biggl[$}; \node[right] at (1,1.75) {$\Biggr] = \delta_Q(x),$};
\end{scope}
\end{tikzpicture}
\end{align*}
which completes the proof.
\end{proof}

\begin{remark}
As mentioned above, we expect that every closable derivation $\delta:A \to A \otimes A^{op}$ has an extension to a derivation $\widetilde \delta:\mc V_+ \to \mc V_+e_0 \otimes e_0\mc V_+^{op}$ which vanishes on $\mc A_+$.  To show this it would be enough to find a system matrix units $\{e_{ij}\}$ in the completion of $\mc V_+$ such that $e_{11} = e_0$, and which are smooth enough that $e_{1i}\mc V_+ e_{j1}$ is contained in a holomorphic closure of $A$.  Once we have such a system, we can define $\widetilde \delta$ by
\begin{equation*}
 \widetilde \delta(a) = \sum_{i,j} e_{i1}\overline \delta(e_{1i}ae_{j1}) e_{1j},
\end{equation*}
where $\overline \delta$ denotes the closure and we use crucially that $e_{1i}ae_{j1} \in \frk D(\overline \delta)$ for any $a \in \mc V_+$.  

The existence of such a system of matrix units should in principle follow from a more careful analysis of the arguments in \cite{gjs2}, and of the work of Dykema \cite{dyk1}, \cite{dyk2}.  Since we are only aiming in this section to provide some intuition for the parameter $r_0$, and its relation to derivations on $A$, we will avoid such technical questions.  But let us point out that we do have natural extensions of the derivations $\delta_Q$ to $\mc V_+$ as follows.
\end{remark}

\begin{proposition}
If $Q \in \Phi$, then there is a derivation $\widetilde \delta:\mc V_+ \to \mc \mc V_+e_0 \otimes e_0\mc V_+^{op}$ such that $\widetilde \delta|_{\mc A_+} = 0$ and $\widetilde \delta|_A = \delta_Q$.
\end{proposition}

\begin{proof}
Let $Q = \sum_{s,t \geq 0} Q_{s,t}$, where $Q_{s,t} \in V_{1,0}^{(-1)^s}(s,t)$.  Define $\widetilde \delta$ by $\widetilde \delta|_{\mc A_+} = 0$ and
\begin{align*}
\widetilde \delta(c_n) &= \sum_{\substack{\text{$s$ even}\\\text{$t$ odd}}} a_{2n,s} \otimes b_{2n+1,t}^{op},\\
\widetilde \delta(c_n^*) &= \sum_{\substack{\text{$s$ odd}\\\text{$t$ even}}} a_{2n+1,s} \otimes b_{2n,t}^{op},
\end{align*}
where $a_{m,s} \otimes b_{m,t}^{op} \in V_{m,0}^{(-1)^s}(s,0) \otimes V_{0,m}^+(t,0)$ are defined by
\begin{align*}
\begin{tikzpicture}
\begin{scope}[yscale=.75]
\draw[Box] (0,1.5) rectangle (1,2.5); \node[scale=.75] at (.5,2) {$a_{2n+1,s}$};
\draw[thickline] (.5,2.5) -- (.5,2.75); \draw[thickline] (0,2) arc(270:180:.25cm) -- (-.25,2.75);
\draw[Box] (0,0) rectangle (1,1); \node[scale=.75] at (.5,.5) {$b_{2n,t}$};
\draw[thickline] (.5,0) -- (.5,-.25); \draw[thickline] (0,.5) arc(90:180:.25cm) -- (-.25,-.25);
\end{scope}
\begin{scope}[xshift=4cm,yshift=.4cm]
\node[left] at (0,.5) {$= E_{M_0 \otimes M_0^{op}}\Biggl[$};
\draw[thickline] (.25,-.25) -- node[count,rectangle,rounded corners] {$2n$} (.25,1.25);
\begin{scope}[xshift=1cm]
\draw[Box] (0,0) rectangle (1,1); \node at (.5,.5) {$Q_{s,t}$};
\draw[thick] (0,.5) arc(-90:-180:.25cm and .5cm) -- (-.25,1.25);
\draw[thickline] (.5,1) -- (.5,1.25);
\draw[thickline] (.5,0) -- (.5,-.25);
\node[right] at (1.25,.5) {$\Biggr],$};
\end{scope}
\end{scope}
\end{tikzpicture}\\
\begin{tikzpicture}
\begin{scope}[yscale=.75]
\draw[Box] (0,1.5) rectangle (1,2.5); \node[scale=.75] at (.5,2) {$a_{2n,s}$};
\draw[thickline] (.5,2.5) -- (.5,2.75); \draw[thickline] (0,2) arc(270:180:.25cm) -- (-.25,2.75);
\draw[Box] (0,0) rectangle (1,1); \node[scale=.75] at (.5,.5) {$b_{2n+1,t}$};
\draw[thickline] (.5,0) -- (.5,-.25); \draw[thickline] (0,.75) arc(90:180:.5cm) -- (-.5,-.25);
\draw[thick] (0,.5) arc(90:270:.25cm and .125cm);
\end{scope}
\begin{scope}[xshift=4cm,yshift=.4cm]
\node[left] at (-.1,.5) {$= E_{M_0 \otimes M_0^{op}}\Biggl[$};
\draw[thickline] (.25,-.25) -- node[count,rectangle,rounded corners,scale=.7] {$2n-1$} (.25,1.25);
\begin{scope}[xshift=1cm]
\draw[Box] (0,0) rectangle (1,1); \node at (.5,.5) {$Q_{s,t}$};
\draw[thick] (0,.5) arc(-90:-180:.25cm and .5cm) -- (-.25,1.25);
\draw[thickline] (.5,1) -- (.5,1.25);
\draw[thickline] (.5,0) -- (.5,-.25);
\node[right] at (1.25,.5) {$\Biggr].$};
\end{scope}
\end{scope}
\end{tikzpicture} 
\end{align*}
It then follows from the proof of Proposition \ref{extension} that $\widetilde \delta$ is a well-defined derivation, and that its restriction to $A$ is $\delta_Q$.
\end{proof}

\begin{remark}
In Voiculescu's free analysis \cite{voi3},\cite{voi4}, the central objects are derivations $\delta:A \to A \otimes A$ which are \textit{coassociative}, i.e.
\begin{equation*}
 (\delta \otimes \mathrm{id}) \circ \delta = (\mathrm{id} \otimes \delta) \circ \delta.
\end{equation*}
Unfortunately, there do not appear to be any natural coassociative derivations among the $\delta_Q$ which we have considered in this section.  

There is however a natural way of defining coassociative derivations on $A$ as follows.  Given $Q \in P_1$, define $\partial_{Q}:A \to A \otimes A$ by
\begin{equation*}
 \begin{tikzpicture}[scale=.75]
  \draw[Box](0,0) rectangle (2,1); \node at (1,.5) {$x$};
  \draw[Box](4,0) rectangle (5,1); \node at (4.5,.5) {$Q$};
  \draw[verythickline] (.5,1) -- (.5,1.75);
  \draw[medthick] (1,1) arc(180:90:.75cm) -- (4,1.75) arc(90:0:.5cm) -- (4.5,1);
  \draw[verythickline] (1.5,1) arc(180:0:.5cm) -- node[count,rectangle,rounded corners] {$2k$} (2.5,-.5);
  \node[left] at (-.25,.5) {$\partial_Q(x) = \displaystyle \sum_{k=0}^n(\mathrm{id} \otimes \mathrm{op})E_{M \otimes M^{op}}\Biggl[$};
  \node[right] at (5.25,.5) {$\Biggr]$};
 \end{tikzpicture}
\end{equation*}
It is easy to see that $\partial_Q$ is coassociative, so that $(A,\partial_Q)$ forms a (graded) \textit{infinitesimal bialgebra} \cite{voi3},\cite{agu}.  However these derivations do not behave well with respect to the Voiculescu trace, in particular they do not appear to be closable.  
\end{remark}

\def\cprime{$'$}

\end{document}